\documentclass[12pt,reqno]{amsart}
\usepackage{amsmath,amssymb}
\usepackage{color}
\usepackage{cite}
\newtheorem{theorem}{Theorem}[section]
\newtheorem{definition}[theorem]{Definition}
\newtheorem{proposition}[theorem]{Proposition}
\newtheorem{lemma}[theorem]{Lemma}

\newtheorem{remark}[theorem]{Remark}

 \makeatletter
 \evensidemargin  \oddsidemargin
 \makeatother

\newcommand{\A}{\mathcal{A}}

 \newcommand{\W}{\mathcal  L}

\begin{document}
\thispagestyle{empty}
 \setcounter{page}{1}

\begin{center}

{\large\bf  KMS states   on  a    generalized  Toeplitz algebra }
\vskip.20in
{\bf  Jieun Ahn  and  Sun Young Jang$^*$  } \\[2mm]

\author{Jieun Ahn}
\address{Department of Mathematics,
\newline  \indent  University of Ulsan, 
\newline  \indent Ulsan 44610, Republic of  Korea}
\email{je302@naver.com}

\author{Sun Young Jang}
 \address{Department of Mathematics,
\newline \indent University of Ulsan,  
\newline \indent Ulsan 44610, Republic of  Korea}
\email{77jsym@gmail.com}

\vskip 5mm

\end{center}

\vskip 5mm
 \noindent{\footnotesize{\bf Abstract.}
 In this paper, we   consider  a  generalized   Toeplitz  algebra  $\mathcal{T} ( \mathrm{P}\rtimes\Bbb N^{\times})$ for a    non-quasi-lattice  ordered   semigroup  $   \mathrm{P}\rtimes\Bbb N^{\times}$ 
where   $    \mathrm{P}\rtimes\Bbb N^{\times}$   is  a  semidirect  product  of    an additive  semigroup  $ \mathrm{P} = \{0, 2, 3, \cdots  \}$ 
   by  a multiplicative   positive  natural  numbers  semigroup $  \Bbb N^{\times}$. And  also  we compute  the  values of  the  KMS state of 
the  natural $C^*$-dynamical  system  $(  \mathcal{T} ( \mathrm{P}\rtimes\Bbb N^{\times}),   \Bbb R,  \sigma ).$

 \vskip.10in
 \footnotetext {2010 Mathematics Subject Classification:  46L05, 47B35, 4703, 47C15}
 
 \footnotetext {Keywords and phrases:  Toeplitz algebra;  quasi-lattice ordered group; covariant isometric representation; KMS-state}

\footnotetext {$^*$Corresponding author}

\baselineskip=14pt

  \newtheorem{df}{Definition}[section]
  \newtheorem{rk}[df]{Remark}
   \newtheorem{lem}[df]{Lemma}
   \newtheorem{thm}[df]{Theorem}
   \newtheorem{pro}[df]{Proposition}
   \newtheorem{cor}[df]{Corollary}
   \newtheorem{ex}[df]{Example}
   \newtheorem{rem}[df]{Remark}
 \setcounter{section}{0}
 \numberwithin{equation}{section}

\vskip .2in

\section{Introduction }

   In  the recent  decades there  has  been    lots  of interest in  $C^*$-algebras generated  by  isometries.  It  seems  that it  started with  
L. A. Coburn's   well-known theorem, which asserted that the $C^*$-algebra generated by a non-unitary  isometry on a separable infinite
dimensional Hilbert space does not depend on the
particular choice of the isometry \cite{Co1, Co}.  Many authors have interests in the generalization of Coburn's
theorem,   the uniqueness property  of the $C^*$-algebras generated by isometries  which  was  called  by  Nica  \cite{Nica}.
  If the $C^*$-algebras
generated by isometries have the uniqueness property, the structures
of those $C^*$-algebras are to some extent  independent of the choice of isometries on a Hilbert space. All the $C^*$-algebras generated by  
isometric representations of  the semigroup $\Bbb N$ of natural
numbers have the uniqueness property and so are isomorphic  to the   classical Toeplitz algebra by Coburn's result.  In addtion,  it was known that the uniqueness property  holds for the $C^*$-algebras generated by one-parameter semigroups of isometries  \cite{Do},
the Cuntz algebras \cite{Cu1},   and   the  $C^*$-algebras  generated  by   isometric  representations  of  the    positive   semigroup  of a totally  ordered  group    \cite{Mur1}.
A. Nica introduced    a  quasi-lattice  ordered  group  which  is  very  suitable  for the uniqueness  property  of  $C^*$-algebras generated by  semigroups of isometries. 
 And Laca  and  Raeburn  also had  important    results  on  it  \cite{LR1, LR2}.
 There  are several  ways to  construct the $C^*$-algebras
generated by semigroups of  isometries.   At  first  Murphy constructed  the  full  semi-group  $C^*$-algebra    by  enveloping all isometric
representations of  a  semi-group $M$  which is  denoted  by $C^*(M)$.
 Seeing from the  definition of the  full semigroup $C^*$-algebra, the  full
semigroup $C^*$-algebra has the universal property as follows:
if we put  the canonical isometric  homomorphism $\mathrm{W}$ of $M$ to
the semigroup $C^*$-algebra $C^*(M)$, then for  any isometric homomorphism $\mathrm{V}$ of $M$
to a unital $C^*$-algebra $B$ there exists
 a unique homomorphism  from $C^*(M)$ to the unital $C^*$-algebra $B$
sending $\mathrm{W}_x$ to $\mathrm{V}_x$ for  each $x \in M$.
    Murphy showed that    $C^*$-algebras  generated  by  isometric  representations of  the  positive  semigroup of a  totally  ordered groups  are  all  isomorphic, 
but  it  seems  that  the  full  semigroup $C^*$-algebra  is   too   big  for  the  uniqueness  property.
 On the  other  hand   we  can  also consider 
the  $C^*$-algebra generated by the left regular isometric representations of   a  left-cancellative
semigroup  $M$,  which 
  has  been studied much  for decades.   We  are going to call it the reduced
semigroup $C^*$-algebra and  denoted it  by  $C^*_{red}(M)$.
As a typical model of the reduced semigroup $C^*$-algebra we have the  classical  Toeplitz algebra
$C^*_{red}(\Bbb N)$  for   the   semigroup $\Bbb  N$ of all
natural numbers.
Nica  defined the  covariant   isometric  representation  for  a quasi-lattice  ordered   semi-group  $ M,$  
 of  which   the left  regular    isometric representation  is   a  typcal  model.   He   also     defined  the  $C^*$-algebras  $C^*_c(M)$ with  the  universal  property  of  covariant  isometric  representations  of  $M$.

J.  Cuntz  and  X. Li    have   improved  the  theory  of   the  $C^*$-algebras  generated  by isometries,      the  theory  of    the  KMS state  of  the  semigroup $C^*$-algebra, 
and  the  amenability  of semigroup  for  the  more  general   semigroups  in  \cite{Cu2, CuX, Xin}.

  Recently   there  are  very  interesting  results  on  the  KMS state  of   $C^*$-dynamical systems  of  $C^*$-algebras  generated  by   isometries  \cite{Cu2, LR3, LF}.
It  is  known  that the Toeplitz-Cuntz  algebra $\mathcal {T}  \mathcal {O}_n$  has  the  KMS states  at  every  inverse  temperature  $\beta \geq \log n $. And   Cuntz  introduced a  
$C^*$-algebra $\mathcal Q _N$ generated  by an isometric  representation of  the  semidirect  product   $\Bbb N  \rtimes \Bbb N^{\times}$   of  the additive  semigroup $\mathbb {N}$  by  the  natural  action   of the  multiplicative  semigroup 
$\mathbb {N}^{\times}$.    He    proved  that $\mathcal Q _N$  is  simple  and    there    exists  a  unique  KMS state at  inverse temperature 1.
  In \cite {LR3} Laca and  Raeburn  investigated  the  structure  of      the semigroup  $C^*$-algebra  $C^*_c(\Bbb N  \rtimes \Bbb N^{\times})$.
 They    showed  that the semigroup  $C^*$-algebra  $C^*_c(\Bbb N  \rtimes \Bbb N^{\times})$   have  interesting  properties  in the   virtue  of   \cite {E, LF, Exel}.
 In particular    they  showed  that  the   KMS state   for the  natural  dynamics  of  $C^*_c( \Bbb N  \rtimes \Bbb N^{\times})$   has phase   transitions.

  In this paper we   consider   the  semidirect  product    $ \mathrm{P}  \rtimes \Bbb  N ^{\times}$   of 
 the  additive  semigroup $\mathrm{P}  =\{0,2,3  \dots \}$ by  the  multiplicative  semigroup   $\Bbb  N ^{\times}.$ 
 The semigroup $\mathrm{P} = \{ 0, 2, 3,  \cdots \}$  is a generating subsemigroup of the
integer group $\Bbb Z$.
   Even  though $(\Bbb Z, \Bbb  N)$ is  the  typical  model  of  a quasi-lattice  ordered  group,    the order structure of $(\Bbb Z, \mathrm{P})$ with the positive cone $\mathrm{P}$  is  
 not  a  quasi-lattice  ordered  group.    The  author  showed  that $ C^*_{red}(\mathrm{P}) $ is  isomorphic to $ C^*(\Bbb N ) $  by  using  Coburn's
result \cite{jang}.
    Even  though      the  semigroup   $ \mathrm{P}\rtimes\Bbb N^{\times}$
gives  a  partial  order on the semi-direct  product  group  $\mathbb{Q} \rtimes \mathbb{Q}^{*}_{+}$,  but ( $\mathbb{Q} \rtimes \mathbb{Q}^{*}_{+}$,  $\mathrm{P}  \rtimes\Bbb  N ^{\times}$) is  not  quasi-lattice  ordered   group.   However 
  we  define  a  covariant  isometric  representation  on $ \mathrm{P}  \rtimes \Bbb  N ^{\times}$ in the similiar  way  of  Nica's   covariant  isometric representation  for  a  quasi-lattice  ordered  group  and   consider    the  $C^*$-algebra   $ \mathcal{T} (\mathrm{P}  \rtimes \Bbb  N ^{\times})$
   generated    by   the  canonical  covariant  isometric representation  on $ \mathrm{P}  \rtimes \Bbb  N ^{\times}$.
     We  get  a  few  results   of  the  $C^*$-algebra  $ \mathcal{T} (\mathrm{P}  \rtimes \Bbb  N ^{\times})$,
  and   we  can also show  how   the  KMS  state   of  the  natural  dynamical  system  $( \mathcal{T} (\mathrm{P}  \rtimes \Bbb  N ^{\times}),  \mathbb R,  \sigma)$  acts.


\section{A Non-quasi-lattice  order  on   $\mathbb{Q} \rtimes \mathbb{Q}^{*}_{+}$}
\label{intro}

Let  $\mathbb{Q} \rtimes \mathbb{Q}^{*}_{+}$ denote the semidirect product of the additive rationals  $\mathbb Q$  by the multiplicative positive rationals $\mathbb{Q}^{*}_{+},$ where the group operation and inverse are given by

$$ (r,x)(s,y) = (r+xs, ~xy)   \hskip  2pc \text  {for} ~ r,s      \in \mathbb{Q}  ~~ \text {and}~~ x,y \in \mathbb{Q}^{*}_{+},  $$ 
$$ (r,x)^{-1}=(-x^{-1}r, ~x^{-1})   \hskip  2pc \text  {for}  ~   r \in \mathbb{Q}   ~~ \text {and} ~~   x \in \mathbb{Q}^{*}_{+}. $$
Let $ \mathrm{P}=  \{0,2,3, \cdots \}$ be  a semigroup  of  $\mathbb Z$.  Then the semidirect  product   $ \mathrm{P} \rtimes \mathbb{N}^{\times}$ is  the subsemigroup  of    $\mathbb{Q} \rtimes \mathbb{Q}^{*}_{+}$.

\begin{proposition} 
The semigroup   $ \mathrm{P}\rtimes \mathbb{N}^{\times}$ is  a  generating semigroup   of   $\mathbb{Q} \rtimes \mathbb{Q}^{*}_{+}$  and  the elements $(2,1), (3,1)$, and $\{ (0,p) :  p~is~prime\}$  satisfy the  relations
$$(0,p)(2,1)=(2,1)^{p}(0,p),~(0,p)(3,1)=(3,1)^{p}(0,p), ~  \text    {and} ~(0,p)(0,q)=(0,q)(0,p)$$
for all  prime  numbers  $ p$  and  $ q.$
\end{proposition}

\begin{proof}
The group $\mathbb{Q} \rtimes \mathbb{Q}^{*}_{+}$ is generated by elements $(1,1)$ and $\{ (0,p) :  p$ is prime\} which satisfy the relations
$$(0,p)(1,1)=(1,1)^{p}(0,p)~\text {and}~ (0,p)(0,q)=(0,q)(0,p)$$
for all  prime  numbers  $ p, q$
and this is a presentation of $\mathbb{Q} \rtimes \mathbb{Q}^{*}_{+}$  in \cite{LR3}.

We shall consider the unital subsemigroup $\mathrm{P} \rtimes \mathbb{N}^{\times} $ of $\mathbb{Q} \rtimes \mathbb{Q}^{*}_{+}$ interpreted in the category of monoids where $\mathrm{P}=\{0,2,3,4,5,\cdots\}$. Since $(2,1)^{-1}(3,1)=(1,1)$   in  
 $\mathbb{Q} \rtimes \mathbb{Q}^{*}_{+}$ , $\mathrm{P} \rtimes \mathbb{N}^{\times}$ can generate $\mathbb{Q} \rtimes \mathbb{Q}^{*}_{+}$.  Furthermore  we  see  that  $\mathrm{P} \rtimes \mathbb{N}^{\times}$is generated by the elements $(2,1),$  $(3,1),$ and $\{ (0,p) :  p~is~prime\}$ which satisfy the relations $(0,p)(2,1)=(2,1)^{p}(0,p)$,~$(0,p)(3,1)=(3,1)^{p}(0,p),$~and~ $(0,p)(0,q)=(0,q)(0,p)$ for all prime numbers   $p,q$.

\end{proof}

 Since $(\mathrm{P} \rtimes \mathbb{N}^{\times}) \cap (\mathrm{P} \rtimes \mathbb{N}^{\times})^{-1} = \{(0,1)\},$ the subsemigroup $\mathrm{P} \rtimes \mathbb{N}^{\times}$ induces a left-invariant partial order on $\mathbb{Q} \rtimes \mathbb{Q}^{*}_{+}$ as follows : for $(r,x)$ and $(s,y)$ in $\mathbb{Q} \times \mathbb{Q}^{*}_{+},$
\begin{eqnarray} \label{eq:order}
 (r,x) \leq (s,y) &\Leftrightarrow &
(r,x)^{-1} (s,y) \in \mathrm{P} \rtimes \mathbb{N}^{\times} \nonumber \\
&\Leftrightarrow& x^{-1}(s-r) \in \mathrm{P} ~\text {and}~ x^{-1}y \in \mathbb{N}^{\times}. \label{9}
\end{eqnarray}

\begin{remark} 
The pair $(\mathbb{Q} \rtimes \mathbb{Q}^{*}_{+},~\mathrm{P}\rtimes \mathbb{N}^{\times})$ is not a quasi-lattice ordered group.
\end{remark}

It is  sufficient   to show that there are two elements  in $\mathrm{P} \rtimes \mathbb{N}^{\times}$   with  common upper bounds in $\mathrm{P} \rtimes \mathbb{N}^{\times}$   
which don't  have   the   least common upper bound in  $\mathrm{P} \rtimes \mathbb{N}^{\times}$.    We  consider  two  elements  $(5, 7)$ and $(2, 3).$
Suppose  that  $(k, c) \in \mathrm{P} \rtimes \mathbb{N}^{\times}$,  $(5, 7)\leq (k, c)$,  and  $(2, 3)\leq (k, c).$
Then from (\ref{eq:order}) we have $k \in 5+7\mathrm{P}$, $k \in 2+3\mathrm{P},$ and $c \in 21\mathbb{N}^{\times}.$
Since \[\begin{cases}
k \equiv 5~(\text{mod}~7) ,& k \neq 12, \\
k \equiv 2~(\text{mod}~3) ,& k \neq 5, 
\end{cases}\]
we  see  that   $k=26, 47, 68, \cdots $and  $ c= 21, 42, 63,  \cdots$.   
If  $(5, 7)$ and $(2, 3)$ have the  least  common upper bound, it  should  be $(26, 21)$ or $(47, 21).$  
But   $(26, 21)$  and $(47, 21)$ are  not comparable  in $\mathbb{Q} \rtimes \mathbb{Q}^{*}_{+}$,   so $(5, 7)$ and $(2, 3)$  do not have their least common upper bound in  $\mathrm{P} \rtimes \mathbb{N}^{\times}$. Therefore  $\mathrm{P} \rtimes \mathbb{N}^{\times}$ is not quasi-lattice ordered group.

  We  will  denote  the smallest one  among  common upper bounds of $(r,x)$ and $(s,y)$  by   $(r,x) \Cup (s,y)$    in the usual order in $\mathbb{Q} \rtimes \mathbb{Q}^{*}_{+}$.

\begin{remark} 
We see  that  two elements $(m,a)$ and $(n,b)$ of $\mathrm{P} \rtimes \mathbb{N}^{\times}$ have a common upper bound if and only if the set $(m+a\mathrm{P})\cap (n+b\mathrm{P})$ is nonempty.
  We see that
\begin{equation*}
(m,a) \Cup (n,b) = \begin {cases}
\infty & \text {if ~$(m+a\mathrm{P})$ $\cap$ $(n+b \mathrm{P})$ = $\emptyset$,}\\
(\ell, lcm(a,b)) & \text {if~$(m+a\mathrm{P})$ $\cap$ $(n+b \mathrm{P})$ $\neq \emptyset$,}\\
\end{cases}
\end{equation*}
where $\ell$ is the smallest element of $(m+a\mathrm{P})\cap (n+b\mathrm{P})$ in the usual order. 
\end{remark}

 The next proposition  shows  how the Euclidean algorithm is related 
  with  the  further discussion  of  this  paper.
 Recall that   $gcd(a,b)$  and  $lcm(c, d)$  is the  greatest  common  divisor  of $a$ and $b$   in  $ \mathbb  N$ and  the  least    common  multiple   of $c$ and $d$    in  $ \mathbb  N$, respectively.

\begin{proposition} \label{pr:cal}
Suppose that $(m,a)$ and $(n,b)$ are in $\mathrm{P} \rtimes \mathbb{N}^{\times}.$ Then \\
$(1)$ $(m+a\mathrm{P}) \cap (n+b\mathrm{P})$ is nonempty if and only if $gcd(a,b) \mid m-n.$ \\
$(2)$ If  $(m+a\mathrm{P}) \cap (n+b\mathrm{P})$is nonempty  and  we  denote  $a^{'}=a/gcd(a,b),$ ~ $b^{'}=b/gcd(a,b),$ and $(\alpha,~\beta)$ is the smallest non-negative solution  of $(n-m)/gcd(a,b) = \alpha a^{'} - \beta b^{'}$ with $\alpha \neq 1$  and  $\beta \neq 1$,
then $\ell := m+a\alpha =n+b\beta$ is the smallest element of $(m+a\mathrm{P}) \cap(n+b\mathrm{P})$ and we have 
$$(m,a) \Cup (n,b) = (\ell, lcm(a,b)),$$
$$(m,a)^{-1}(\ell, lcm(a,b))=(a^{-1}(\ell -m), a^{-1}lcm(a,b))=(\alpha, b^{'}), \text{and}$$
$$(n,b)^{-1}(\ell, lcm(a,b))=(b^{-1}(\ell -n), b^{-1}lcm(a,b))=(\beta, a^{'}).$$
\end{proposition}

\begin{proof}
$(m+a\mathrm{P}) \cap (n+b\mathrm{P}) \neq \emptyset \Longleftrightarrow (m+a\mathbb{Z}) \cap (n+b\mathbb{Z}) \neq \emptyset \Longleftrightarrow m \equiv n$ $(\!\!\!\!\mod gcd(a,b)).$ Then every solution of $(n-m)/gcd(a,b)=\alpha a^{'}-\beta b^{'}$ satisfies $m+a\alpha = n+b\beta$ and the smallest non-negative solution of $(n-m)/gcd(a,b)=\alpha a^{'}-\beta b^{'}$ gives  the smallest common value. The rest two formulas  can  be  gotten   by  the  easy calculation.
\end{proof}

\section{The generalized  Toeplitz  algebra   $  {\mathcal{T}}(\mathrm{P} \rtimes \mathbb{N}^{\times})$}
\label{intro}

   In  this section,  we  construct  the  $C^*$-algebra  $\mathcal{T} (\mathrm{P} \rtimes \mathbb{N}^{\times})$   generated  by   an  isometric  representation  of $\mathrm{P} \rtimes \mathbb{N}^{\times}$ and    analyze its  structure   by    thoughts    from \cite{LR1}.
First,  we  introduce    the  isometric  representation of  a  discrete semigroup  $M$; 

let  $M$   denote a
 semigroup with unit $e$ and
  $\mathcal B$ be a unital $C^*$-algebra.  A map
$\mathrm{W}: M \to \mathcal B, x \mapsto \mathrm{W}_x$  is called an {\it isometric homomorphism} if
 $~~\mathrm{W}_e = 1$, $\mathrm{W}_x$ is an isometry and
$\mathrm{W}_{xy} = \mathrm{W}_x\mathrm{W}_y$ for all $x, y \in M.$
If $\mathcal B$ is the $*$-algebra $ {\mathcal B}(H)$ of all bounded linear
operators of a non-zero Hilbert space $H$,
 we call $(H,\mathrm{W})$ an  {\it isometric representation} of $M$.

Nica  introduced the  covariant isometric  representation of  a quasi-lattice  ordered  group  as  follows:
for a  quasi-lattice  ordered  group $M$  an  isometric  representation $V  : M  \to {\mathcal B}(H)$  is Nica  {\it covariant} if
$$V_{x} V^{*}_{x}   V_yV^*_y\\
 = \begin{cases}
0 & if~~~x \vee y  =\infty,\\
V_{x \vee y}V^*_{x \vee y} &  if ~~~x \vee y  < \infty 
\end{cases}
$$
where $x \vee y$ is the least common upper bound of $x$ and $y$ in $M.$
  It  is  known  that  Nica's  covariance     is  a  very  suitable  isometric  representation  to  explain the uniqueness  property  of  $C^*$-algebras generated  by isometric  representations. 
Even  though $(\mathbb{Q} \rtimes \mathbb{Q}^{*}_{+},~\mathrm{P}\rtimes \mathbb{N}^{\times})$ is  not  quasi-lattice  ordered  group, we  can  define the  covariant isometric  representaion  of  $\mathrm{P}\rtimes \mathbb{N}^{\times}$  in the  sence  of  Nica's   covariant  isometric  representation.

\begin{definition} 
A isometric  representation  $\mathrm{W} : \mathrm{P} \rtimes \mathbb{N} ^{\times} \to  \mathcal{B}(\mathrm{H})$  of   $\mathrm{P} \rtimes \mathbb{N} ^{\times}$  
on a Hilbert space   $H$ is     covariant if it satisfies
\begin{equation} 
\mathrm{W}_{(m,a)}\mathrm{W}^{*}_{(m,a)}\!\mathrm{W}_{(n,b)}\mathrm{W}^{*}_{(n,b)}\\
 = \begin{cases}
0 & if ~(m+a\mathrm{P})\cap(n+b \mathrm{P}) = \emptyset,\\
\!\mathrm{W}_{(m,a)\Cup(n,b)}\mathrm{W}^{*}_{(m,a)\Cup(n,b)}\!\!\!\!&if ~(m+a\mathrm{P})\cap(n+b \mathrm{P}) \neq \emptyset.
\end{cases}
\end{equation}
\end{definition}
 We use the notation  $\mathrm{W}_{\infty}=0 $ when  $(m,a)  \Cup (n,b) = {\infty}, $  thus we can   always  write  
$$\mathrm{W}_{(m,a)}\mathrm{W}_{(m,a)}^{*}\mathrm{W}_{(n,b)}\mathrm{W}_{(n,b)}^{*}=\mathrm{W}_{(m,a)\Cup(n,b)}\mathrm{W}^{*}_{(m,a)\Cup(n,b)}$$
 for all   $(m,a), (n,b)\in \mathrm{P} \rtimes \mathbb{N} ^{\times}.$
With this convention, the covariant condition is equivalent to 
\begin{equation}   \label{def:nica}
\mathrm{W}_{(m,a)}^{*}\mathrm{W}_{(n,b)} = \mathrm{W}_{(m,a)^{-1}\sigma}\mathrm{W}_{(n,b)^{-1}\sigma}^{*}
\end{equation}
  for all   $(m,a), (n,b)\in \mathrm{P} \rtimes \mathbb{N} ^{\times}~\text{where}~\sigma= (m,a) \Cup (n,b)$.

The motivation  of  the   condition  of   the covariant   isometric  representation is  the range projections of  the left regular isometric  representation of  a  semigroup  $M$.  Nica  called  it  {\it the Wiener-Hopf representation.} 
  The  left  regular isometric representation on  the  discrete  semigroup $M$   is given by
\begin{eqnarray*}
\mathcal{L}_{m}\delta_{n} = \delta_{mn}  \hskip  1pc   \text  {for}  ~~ m, n  \in  M 
\end{eqnarray*}
where~ $\{\delta_{n}: n \in   M  \}$ is the canonical orthonormal basis for $\ell^{2}(M).$  
The  left  regular isometric representation on $\ell^{2}(\mathrm{P} \rtimes \mathbb{N} ^{\times})$ is similarily  defined  as  follows:
\begin{eqnarray*}
\mathcal{L}_{(m,a)}\delta_{(n,b)} = \delta_{(m,a)(n,b)}  \hskip  1pc   \text  {for}  ~~(m,a),(n,b)\in \mathrm{P} \rtimes \mathbb{N} ^{\times}
\end{eqnarray*}
where  $\{\delta_{(n,b)} :(n,b) \in \mathrm{P} \rtimes \mathbb{N} ^{\times}\}$ is the canonical orthonormal basis for $\ell^{2}(\mathrm{P} \rtimes \mathbb{N} ^{\times}).$  
  Though  the $C^{*}$-algebra generated by the  left  regular  isometric representation  is  called  in  several    ways,
   we  call  it  the reduced  semigroup  $C^{*}$-algebra and denote  it  by $\mathcal{C}_{red}(\mathrm{P} \rtimes \mathbb{N} ^{\times}).$

 By  similar way   in \cite{LR1}  we  can  have   another semigroup $C^*$-algebra generated  by a  covariant  isometric  representation of 
$ \mathrm{P} \rtimes \mathbb{N} ^{\times}$.

\begin{definition} 
 The universal $C^*$-algebra for covariant isometric representations of ~$ \mathrm{P} \rtimes \mathbb{N} ^{\times},$ denoted by 
$ \mathcal{T}(\mathrm{P} \rtimes \mathbb{N} ^{\times})$, is the $C^{*}$-algebra generated by   the canonical covariant isometric representation $\mathrm{W}: \mathrm{P} \rtimes \mathbb{N} ^{\times} \to  \mathcal{T}(\mathrm{P} \rtimes \mathbb{N} ^{\times})$ with the following proprety :
if $~\mathrm{X}$ is a covariant isometric representation of $\mathrm{P} \rtimes \mathbb{N} ^{\times}$,  then there is a homomorphism $\pi :  \mathcal{T} (\mathrm{P} \rtimes \mathbb{N} ^{\times}) \to C^{*}(\{\mathrm{X}_{(m,a)}:(m,a)\in  \mathrm{P} \rtimes \mathbb{N} ^{\times}\})$ such that $\pi(\mathrm{W}_{(m,a)})=\mathrm{X}_{(m,a)}.$
We call    $ \mathcal{T}(\mathrm{P} \rtimes \mathbb{N} ^{\times})$ the generalized  Toeplitz  algebra of $\mathrm{P} \rtimes \mathbb{N} ^{\times}$.
\end{definition}

\begin{proposition} \label{def:uni}
Let $\mathrm{W}:\mathrm{P} \rtimes \mathbb{N} ^{\times} \to\mathcal{B}(\ell^{2}(\mathrm{P} \rtimes \mathbb{N} ^{\times}))$ be the canonical  covariant  isometric representaton of $\mathrm{P} \rtimes \mathbb{N} ^{\times}$ and  
$  \mathcal{T}(\mathrm{P} \rtimes \mathbb{N} ^{\times})$ be the universal $C^{*}$-algebra generated by the canonical  covariant  isometric     representation  $W$.  Then the linear span of $\{\mathrm{W}_{(m,a)}\mathrm{W}_{(n,b)}^{*}:(m,a),(n,b) \in \mathrm{P} \rtimes \mathbb{N} ^{\times} \}$ is a dense $*$-subalgebra of  $ \mathcal{T}(\mathrm{P} \rtimes \mathbb{N} ^{\times}).$
\end{proposition}

\begin{proof}
It is enough to show that $\mathrm{W}_{s_{1}}\mathrm{W}^{*}_{t_{1}}\mathrm{W}_{s_{2}}\mathrm{W}^{*}_{t_{2}}\cdots \mathrm{W}_{s_{n}}\mathrm{W}^{*}_{t_{n}}\mathrm{W}_{s_{n+1}}$ can be reduced to  $\mathrm{W}_{s}\mathrm{W}^{*}_{t}$ for $s_{i},~t_{i},~s,~t\in \mathrm{P} \rtimes \mathbb{N} ^{\times}.$ If $(m,a)$ and $(n,b)$ have $\sigma =(m,a)\Cup(n,b) \in \mathrm{P} \rtimes \mathbb{N} ^{\times}, $ then
\begin{eqnarray*}
\mathrm{W}^{*}_{(m,a)}\mathrm{W}_{(n,b)}&=&\mathrm{W}^{*}_{(m,a)}(\mathrm{W}_{(m,a)}\mathrm{W}^{*}_{(m,a)}\mathrm{W}_{(n,b)}\mathrm{W}^{*}_{(n,b)})\mathrm{W}_{(n,b)}\\
&=&\mathrm{W}^{*}_{(m,a)}\mathrm{W}_{\sigma}\mathrm{W}^{*}_{\sigma}\mathrm{W}_{(n,b)}\\
&=&\mathrm{W}^{*}_{(m,a)}\mathrm{W}_{(m,a)}\mathrm{W}_{(m,a)^{-1}\sigma}\mathrm{W}^{*}_{(n,b)^{-1}\sigma}\mathrm{W}^{*}_{(n,b)}\mathrm{W}_{(n,b)}\\
&=&\mathrm{W}_{(m,a)^{-1}\sigma}\mathrm{W}^{*}_{(n,b)^{-1}\sigma}~.
\end{eqnarray*}
If $(m,a)$ and $(n,b)$ don't have $(m,a)\Cup(n,b) $ in $\mathrm{P} \rtimes \mathbb{N} ^{\times},$ we have
$$\mathrm{W}_{(m,a)}\mathrm{W}^{*}_{(m,a)}\mathrm{W}_{(n,b)}\mathrm{W}^{*}_{(n,b)}=0.\\$$
 So  we  can  see 
\begin{eqnarray*}
\mathrm{W}_{(m,a)}\mathrm{W}^{*}_{(n,b)}\mathrm{W}_{(s,t)}\mathrm{W}^{*}_{(u,v)}&=&\mathrm{W}_{(m,a)}\mathrm{W}_{(n,b)^{-1}\sigma}\mathrm{W}^{*}_{(s,t)^{-1}\sigma}\mathrm{W}^{*}_{(u,v)}\\
&=&\mathrm{W}_{(m,a)(n,b)^{-1}\sigma}\mathrm{W}^{*}_{(u,v)(s,t)^{-1}\sigma}\end{eqnarray*}
   where $\sigma~=~(n,b)\Cup (s,t) \in \mathrm{P} \rtimes \mathbb{N} ^{\times},$
 it  follows  that we  can  reduce  $\mathrm{W}_{s_{1}}\mathrm{W}^{*}_{t_{1}}\mathrm{W}_{s_{2}}\mathrm{W}^{*}_{t_{2}}\cdots \mathrm{W}_{s_{n}}\mathrm{W}^{*}_{t_{n}}\mathrm{W}_{s_{n+1}}$   to  $\mathrm{W}_{s}\mathrm{W}^{*}_{t}.$
\end{proof}

\begin{theorem} \label{th:T6} 
Let $\mathcal{A}$ be the universal $C^{*}$-algebra generated by isometries $s$,  $t$, and $\{\upsilon_{p}~:~p ~is~ prime\}$ satisfying the relations\\
$(\mathrm{R1})$~~~$t^{2}=s^{3}$,\\
$(\mathrm{R2})$~~~$ts=st,$~~$s^{*}t=ts^{*},$ and~~$t^{*}s=st^{*}$,\\
$(\mathrm{T1})$~~~$\upsilon_{p}s=s^{p}\upsilon_{p},$~~$\upsilon_{p}t=t^{p}\upsilon_{p},$ ~~$\upsilon_{p}s^{*}=s^{*p}\upsilon_{p},$ \text {and}~~$\upsilon_{p}t^{*}=t^{*p}\upsilon_{p}$, \\
$(\mathrm{T2})$~~~$\upsilon_{p}\upsilon_{q}=\upsilon_{q}\upsilon_{p}$,\\
$(\mathrm{T3})$~~~$\upsilon_{p}^{*}\upsilon_{q}=\upsilon_{q}\upsilon_{p}^{*}$ when $p\neq q$,\\
$(\mathrm{T4})$~~~$s^{*}\upsilon_{p}=s^{p-1}\upsilon_{p}s^{*}$,\\
$(\mathrm{T5})$~~~$\upsilon_{p}^{*}s^{k_{1}}t^{k_{2}}\upsilon_{p}=0$~for~ $1\leq 2k_{1}+3k_{2}<p$, \\
$(\mathrm{T6})$~~~$\upsilon_{p}^{*}s^{k}\upsilon_{p}=0$~for ~$1 \leq k <p$~when~ $p \neq2,$ and~~ $\upsilon_{2}^{*}s\upsilon_{2} =ts^{*},$
$\upsilon_{p}^{*}t^{k}\upsilon_{p}=0$~for ~$1\leq k <p$~when~$ p \neq3,$ 
$\upsilon_{3}^{*}t\upsilon_{3}=ts^{*},$    $~\upsilon_{3}^{*}t^{2}\upsilon_{3}=s,$  and  $\upsilon_{p}^{*}t^{k}s^{*k}\upsilon_{p}=0$~for~ $1\leq k<p$.\\
Then there is a homomorphism $\rho_{\mathrm{W}}$ of $\mathcal{A}$ into  $\mathcal{T}(\mathrm{P} \rtimes \mathbb{N} ^{\times})$ such that
 $\rho_{\mathrm{W}}(s)=\mathrm{W}_{(2,1)},~\rho_{\mathrm{W}}(t)=\mathrm{W}_{(3,1)}$,  and $\rho_{\mathrm{W}}(\upsilon_{p})=\mathrm{W}_{(0,p)}$ for every prime $p.$
\end{theorem}
\begin{proof}
We  put  $\mathrm{S}=\mathrm{W}_{(2,1)},\mathrm{T}=\mathrm{W}_{(3,1)},$ and $\mathrm{V}_{p}=\mathrm{W}_{(0,p)}$. 
We  will  show  that $\mathrm{S, T},$ and $\mathrm{V}_p$ for  a  prime  $p$  satisfy    $(\mathrm{R1, R2})$  and $(\mathrm{T1 - T6}).$  
  It  is  easily  shown that $ \mathrm{T^2 = S^3},$ $\mathrm{TS  =  ST}$, $\mathrm{V}_{p}\mathrm{S}=\mathrm{S}^{p}\mathrm{V}_{p}$, $\mathrm{V}_{p}\mathrm{T}=\mathrm{T}^{p}\mathrm{V}_{p},$  and $\mathrm{V}_{p}\mathrm{V}_{q}=\mathrm{V}_{q}\mathrm{V}_{p}$ hold    by  the  definition  of  semi-direct  product.  
  
    Equations $\mathrm{S^*T = TS^*},$ $\mathrm{T^*S = ST^*},$ $\mathrm{V}_{p}\mathrm{S}^* = \mathrm{S}^{*p}\mathrm{V}_{p},$ $\mathrm{V}_{p}\mathrm{T}^* = \mathrm{T}^{*p}\mathrm{V}_{p},$ $\mathrm{V}_{p}^*\mathrm{V}_{q} = \mathrm{V}_{q}\mathrm{V}_{p}^*,$ and $ \mathrm{S}^{*}\mathrm{V}_{p}=\mathrm{S}^{p-1}\mathrm{V}_{p}\mathrm{S}^*$   hold  by  the  virtue of  the Nica covariance relation (\ref{def:nica}) for $(m,a)=(2,1)$ and $(n,b)=(3,1)$ ; for $(m,a)=(3,1)$ and $(n,b)=(2,1)$ ; for $(m,a)=(2p,1)$ and $(n,b)=(0,p)$ ; for $(m,a)=(3p,1)$ and $(n,b)=(0,p)$ ; for $(m,a)=(0,p)$ and $(n,b)=(0,q)$ ; and  for  $(m,a)=(2,1)$ and $(n,b)=(0,p),$ respectively.
  
To prove (T5) we  will  show that $\mathrm{V}^{*}_{p}\mathrm{S}^{k_{1}}\mathrm{T}^{k_{2}}\mathrm{V}_{p}=0$ for $1\leq 2k_{1}+3k_{2}<p.$ 
 Since $(2,1)^{k_{1}}(3,1)^{k_{2}}(0,p)=(2k_{1},1)(3k_{2},1)(0,p)=(2k_{1}+3k_{2},p), $ we can use the Nica covariance relation (\ref{def:nica}) for  $(m,a)=(0,p)$ and $(n,b)=(2k_{1}+3k_{2},p).$ Then  we  have 
\begin{eqnarray}
\begin{cases} \label{3.2}
\ell \equiv0(mod~p)~but~ \ell \neq p,   \\
\ell \equiv 2k_{1}+3k_{2}(mod~p)~but~ \ell \neq 2k_{1}+3k_{2}+p.  
\end{cases}
\end{eqnarray}
where $(o,p) \Cup (2k_{1}+3k_{2}, p)=(\ell , p).$
By (\ref{3.2})   we  have $2k_{1}+3k_{2}\equiv 0~(mod~p)$, which is contradictory to $1 \leq 2k_{1}+3k_{2} <p.$ This implies that $(p\mathrm{P})\cap((2k_{1}+3k_{2})+p\mathrm{P})=\emptyset $.  
Therefore  it  leads  that $~\mathrm{V}_{p}^{*}\mathrm{S}^{k_{1}}\mathrm{T}^{k_{2}}\mathrm{V}_{p}=0~$ for $1 \leq 2k_{1}+3k_{2} <p.$

  Equations $\mathrm{V}^{*}_{p}\mathrm{S}^{k}\mathrm{V}_{p}=0$ for $1\leq k<p$ when $p \not = 2,$  $\mathrm{V}^{*}_{2}\mathrm{S}\mathrm{V}_{2}=\mathrm{T}\mathrm{S}^{*},$ $\mathrm{V}^{*}_{p}\mathrm{T}^{k}\mathrm{V}_{p}=0$ for $1\leq k<p$ when $p \not = 3,$  $\mathrm{V}^{*}_{3}\mathrm{T}\mathrm{V}_{3}=\mathrm{T}\mathrm{S}^{*},$ and  $\mathrm{V}^{*}_{3}\mathrm{T}^{2}\mathrm{V}_{3}=\mathrm{S}$
    can  be   also  proved  by  the Nica covariance relation (\ref{def:nica}) for  $(m,a)=(0,p)$ and $(n,b)=(2k,p)$; for  $(m,a)=(0,2)$ and $(n,b)=(2,2)$; for  $(m,a)=(0,p)$ and $(n,b)=(3k,p)$; for  $(m,a)=(0,3)$ and $(n,b)=(3,3)$; and for  $(m,a)=(0,3)$ and $(n,b)=(6,3)$, respectively.

 Finally,  we show that $\mathrm{V}^{*}_{p}\mathrm{T}^{k}\mathrm{S}^{*k}\mathrm{V}_{p}=0$ for $1\leq k<p.$
By the Nica covariance condition  $\mathrm{W}^{*}_{(0,p)}\mathrm{W}_{(3k,1)}=\mathrm{W}_{(0,p)^{-1}\sigma_{1}}\mathrm{W}^{*}_{(3k,1)^{-1}\sigma_{1}}$ where $\sigma_{1}=(0,p)\Cup(3k,1)=(\ell,p).$   So  we  have 
 \begin{eqnarray*}
  \mathrm{W}^{*}_{(0,p)}\mathrm{W}^{k}_{(3,1)}\mathrm{W}^{*k}_{(2,1)}\mathrm{W}_{(0,p)} &=& \mathrm{W}_{(0,p)^{-1}\sigma_{1}}\mathrm{W}^{*}_{(3k,1)^{-1}\sigma_{1}}\mathrm{W}^{*k}_{(2,1)}\mathrm{W}_{(0,p)}\\
&=&\mathrm{W}_{(0,p)^{-1}\sigma_{1}}\mathrm{W}^{*}_{(k,1)^{-1}\sigma_{1}}\mathrm{W}_{(0,p)}\\
 &=&\mathrm{W}_{(0,p)^{-1}\sigma_{1}}\mathrm{W}_{\sigma^{-1}_{1}(k,1)\sigma_{2}}\mathrm{W}^{*}_{(0,p)^{-1}\sigma_{2}}
 \end{eqnarray*}
where 
$\sigma_{2}=(k,1)^{-1}\sigma_{1}\Cup(0,p) =(-k+\ell,p)\Cup(0,p)  = (\ell^{'},p).$
Since $\ell^{'}$ is the smallest element of $((-k+\ell)+p\mathrm{P})\cap(p\mathrm{P}),$  we  get
\begin{eqnarray}
\begin{cases} \label{3.3}
\ell^{'} \equiv(-k+\ell)~(mod~p)~but~ \ell^{'} \neq -k+\ell+p,    \\
\ell^{'} \equiv 0~(mod~p)~but~ \ell^{'} \neq p.   
\end{cases}
\end{eqnarray}
By (\ref{3.3})  $k \equiv \ell \equiv 0~(mod~p)$.    It  is   contradictory   to  that   $1\leq k<p.$
Therefore   we  have  $((-k+\ell)+p\mathrm{P})\cap(p\mathrm{P})= \phi$  and  $~\mathrm{W}^{*}_{(0,p)}\mathrm{W}^{k}_{(3,1)}\mathrm{W}^{*k}_{(2,1)}\mathrm{W}_{(0,p)}=0$ for $1\leq k<p.$\\
\end{proof}

\begin{remark} 
~We write $\mathrm{s}^{((k))}$ to mean $\mathrm{s}^{k}$ where $k \geq 0$ and $s^{*(-k)}$ when $k<0, $~~then $\mathrm{s}^{((a+b))}=\mathrm{s}^{((a))}\mathrm{s}^{((b))}.$
\end{remark}

\begin{lemma}   \label{le:k1} 
Let  $\mathcal{A}$   be the  $C^*$-algebra  in  Theorem \ref {th:T6}.
Suppose that   $s,~t$, and $\{\upsilon_{p}:p~is~prime\}$ are isometries satisfying the relations $(\mathrm{R1, R2})$  and $(\mathrm{T1-T6}).$ Then the isometries $s,~t,$ and $~\upsilon_{p}$ for a prime number $p$ satisfy
$$~~\upsilon_{p}^{*}s^{((k_{1}))}t^{((k_{2}))}\upsilon_{p}=0   \hskip  1pc 
  \text{for}~ 1\leq 2k_{1}+3k_{2}<p, ~~k_{1},k_{2} \in \mathbb{Z}. $$
\end{lemma}

\begin{proof}
Let $k=2k_{1}+3k_{2}$. If $(k^{'}_{1},k^{'}_{2})$ is the particular solution of the indeterminate equation $k=2k_{1}+3k_{2}$,  then
$k_{1}= k^{'}_{1}+3u$  and  $k_{2}= k^{'}_{2}-2u$  for $ u \in \mathbb{Z}$
  are  the general solutions. By (R1) and (R2)
\begin{eqnarray*}
\upsilon_{p}^{*}s^{((k_{1}))}t^{((k_{2}))}\upsilon_{p}
&=&\upsilon_{p}^{*}s^{((k^{'}_{1}))}s^{((3u))}t^{*((2u))}t^{((k^{'}_{2}))}\upsilon_{p}\\
&=&\upsilon_{p}^{*}s^{((k^{'}_{1}))}t^{((k^{'}_{2}))}\upsilon_{p}.
\end{eqnarray*}
Therefore  we only consider the particular solution of the indeterminate equation. 
If $k=1$, we can take $k_{1}=-1$ and $ k_{2}=1$.  Then $\upsilon_{p}^{*}s^{*}t\upsilon_{p}=0$ by (T6).
If $k\not=1$, we can take $k_{1}\geq 0 $  and $ k_{2}\geq0$ such that $k=2k_{1}+3k_{2}$.
 Thus  we  have  $\upsilon_{p}^{*}s^{((k_{1}))}t^{((k_{2}))}\upsilon_{p}=\upsilon_{p}^{*}s^{k_{1}}t^{k_{2}}\upsilon_{p}=0~for~ 1\leq 2k_{1}+3k_{2}<p$ by (T5).\end{proof}

\begin{lemma}  Let  $\mathcal{A}$   be the  $C^*$-algebra  in  Theorem \ref {th:T6}.
Suppose that   $s,~t,$ and $\{\upsilon_{p}:p~is~prime\}$ are isometries satisfying the relations  $(\mathrm{R1, R2})$  and $(\mathrm{T1 - T6}).$   Then the isometries $\upsilon_{a} := \Pi_{p}\upsilon_{p}^{e_{p}(a)}$ for $a\in\mathbb{N}^{\times}$ and a  prime  number  $p$  satisfy\\
$(\mathrm{T1}^{'})$~~~$\upsilon_{a}s=s^{a}\upsilon_{a},$~~$\upsilon_{a}t=t^{a}\upsilon_{a},$ ~~$\upsilon_{a}s^{*}=s^{*a}\upsilon_{a},$~and~~$\upsilon_{a}t^{*}=t^{*a}\upsilon_{a}$, \\
$(\mathrm{T2}^{'})$~~~$\upsilon_{a}\upsilon_{b}=\upsilon_{b}\upsilon_{a}$,\\
$(\mathrm{T3}^{'})$~~~$\upsilon_{a}^{*}\upsilon_{b}=\upsilon_{b}\upsilon_{a}^{*}$ whenever $gcd(a,b)=1$,\\
$(\mathrm{T4}^{'})$~~~$s^{*}\upsilon_{a}=s^{a-1}\upsilon_{a}s^{*}$,\\
$(\mathrm{T5}^{'})$~~~$\upsilon_{a}^{*}s^{k_{1}}t^{k_{2}}\upsilon_{a}=0 $ for $1 \leq 2k_{1}+3k_{2} <a$,\\
$(\mathrm{T6}^{'})$~~~$\upsilon_{a}^{*}s^{k}\upsilon_{a}=0 $ for $1 \leq k <a$~when~$a\neq2,$ $\upsilon_{a}^{*}t^{k}\upsilon_{a}=0 $ for $1 \leq k <a$ ~when~$a\neq3,$  and

$~~~~~$$~~~~~$$~~~~~$$~~~~~$$\upsilon_{a}^{*}t^{k}s^{*k}\upsilon_{a}=0 $ for $1 \leq k <a$ for $a\in\mathbb{N^{\times}}.$
\end{lemma}

\begin{proof}
Equations $(\mathrm{T1}^{\prime}),~(\mathrm{T2}^{\prime}),$ and $(\mathrm{T3}^{\prime})$ follow immediately from their counterparts for a prime. We   will  prove  $(\mathrm{T4}^{\prime})$ by induction on the number of prime factors of $ a$.

Suppose that   $(\mathrm{T4}^{\prime})$ is true for every $a\in\mathbb{N}^{\times}$ with $n$ prime factors  and $b=aq\in\mathbb{N}^{\times} $ has $n+1$ prime factors. Then  we  have
\begin{eqnarray*}
 s^{*}\upsilon_{b} = s^{*}\upsilon_{aq}
  =s^{a-1}\upsilon_{a} s^{q-1}\upsilon_{q}s^{*} 
  =s^{a-1}s^{a(q-1)}\upsilon_{a} \upsilon_{q}s^{*} 
  = s^{aq-1}\upsilon_{aq}s^{*} 
  =s^{b-1}\upsilon_{b}s^{*}. 
\end{eqnarray*}
  Therefore we have proved  $(\mathrm{T4}^{\prime}).$ 

For  $(\mathrm{T5}^{\prime}),$   we first prove by induction on $n$
 that $\upsilon_{p}^{*n}s^{k_{1}}t^{k_{2}}\upsilon_{p}^{n} \neq 0$ implies $p^{n} | 2k_{1}+3k_{2}.$ 
Let $n=1$. If $\upsilon_{p}^{*}s^{k_{1}}t^{k_{2}}\upsilon_{p} \neq 0$, then $ p \leq 2k_{1}+3k_{2}$ by (T5).
 Thus we can have $2k_{1}+3k_{2}=pm+\alpha$~~(some ~$m\in \mathbb{N}$~and ~~$0 \leq \alpha < p).$   If  we  put   $m=2m_{1}+3m_{2}$    and $ \alpha=2\alpha_{1}+3\alpha_{2}$, then we can have  $k_{1}=pm_{1}+\alpha_{1}$~and~~$k_{2}=pm_{2}+\alpha_{2} $ which is the particular solution of the equation $2k_{1}+3k_{2}=pm+\alpha.$
Therefore   we  get 
\begin{eqnarray*}
0 &\not=& \upsilon_{p}^{*}s^{k_{1}}t^{k_{2}}\upsilon_{p}\\
&=&(\upsilon_{p}^{*}s^{((pm_{1}))})s^{((\alpha_{1}))}t^{((\alpha_{2}))}(t^{((pm_{2}))}\upsilon_{p})\\
&=&s^{((m_{1}))}(\upsilon_{p}^{*}s^{((\alpha_{1}))}t^{((\alpha_{2}))}\upsilon_{p})t^{((m_{2}))}.
\end{eqnarray*}
By Lemma \ref {le:k1} if  $\upsilon_{p}^{*}s^{((\alpha_{1}))}t^{((\alpha_{2}))}\upsilon_{p}\not=0$ ~~for  ~$0 \leq 2\alpha_{1}+3\alpha_{2} < p$,   then  $\alpha =0.$   So  we  have  a  conclusion  that  $p | 2k_{1}+3k_{2}.$ 

Suppose that ($\mathrm{T5}^{'}$)  holds   for  $n$. We are going to show that the property holds in the case of  $n+1.$
If $\upsilon_{p}^{*(n+1)}s^{((k_{1}))}t^{((k_{2}))}\upsilon_{p}^{(n+1)}=\upsilon_{p}^{*}(\upsilon_{p}^{*n}s^{((k_{1}))}t^{((k_{2}))}\upsilon_{p}^{n})\upsilon_{p} \not= 0,$ then $p^{n} | 2k_{1}+3k_{2}.$ Thus we can take $p^{n}u=2k_{1}+3k_{2}$ (some $u \in \mathbb{N}$). If $u=1,$ then we can take $k_{1}=3k_{0}-p^{n}$ and $k_{2}=p^{n}-2k_{0}$ some $k_{0} \in \mathbb{Z}$ because $2k_{1}+3k_{2}=-2p^{n}+3p^{n}.$    Therefore we have
\begin{eqnarray*}
0 &\not=& \upsilon_{p}^{*}(\upsilon_{p}^{*n}s^{k_{1}}t^{k_{2}}\upsilon^{n}_{p})\upsilon_{p}\\
&=& \upsilon_{p}^{*}(\upsilon_{p}^{*n}s^{*p^{n}})(s^{((3k_{0}))}t^{*((2k_{0}))})(t^{p^{n}}\upsilon^{n}_{p})\upsilon_{p}\\
&=& \upsilon_{p}^{*}(s^{*}\upsilon_{p}^{*n}\upsilon^{n}_{p}t)\upsilon_{p}\\
&=& \upsilon_{p}^{*}s^{*}t\upsilon_{p}.
\end{eqnarray*}
But it is contradiction to ($\mathrm{T6}$). If $u >1,$ we can take $u=2u_{1}+3u_{2}$ some $u_{1}, u_{2} \in \mathbb{N}.$ Thus we can have $k_{1}=u_{1}p^{n}+3k_{0}$ and $k_{2}=u_{2}p^{n}-2k_{0}$ some $k_{0} \in \mathbb{Z}$ because $2k_{1}+3k_{2}=2u_{1}p^{n}+3u_{2}p^{n}.$  Hence  it  leads  that 
\begin{eqnarray*}
0 &\not=& \upsilon_{p}^{*}(\upsilon_{p}^{*n}s^{k_{1}}t^{k_{2}}\upsilon^{n}_{p})\upsilon_{p}\\
&=& \upsilon_{p}^{*}(\upsilon_{p}^{*n}s^{u_{1}p^{n}})(s^{((3k_{0}))}t^{*((2k_{0}))})(t^{u_{2}p^{n}}\upsilon^{n}_{p})\upsilon_{p}\\
&=& \upsilon_{p}^{*}(s^{u_{1}}\upsilon_{p}^{*n}\upsilon^{n}_{p}t^{u_{2}})\upsilon_{p}\\
&=& \upsilon_{p}^{*}s^{u_{1}}t^{u_{2}}\upsilon_{p}.
\end{eqnarray*}
  So  we  have  $p | u$ and $p^{n}u=2k_{1}+3k_{2}.$ These  imply that   $p^{n+1} | 2k_{1}+3k_{2}.$ 
It  follows  that 
\begin{eqnarray*}
\upsilon _{a}^{*}s^{k_{1}}t^{k_{2}}\upsilon_{a}\neq0 &\Rightarrow&\upsilon_{p}^{*e_{p}(a)}s^{k_{1}}t^{k_{2}}\upsilon_{p}^{*e_{p}(a)}\neq 0~ \text{for~ all}~ p|a\\
&\Rightarrow&p^{e_{p}(a)}|2k_{1}+3k_{2}~\text{for~ all}~p|a\\
&\Rightarrow&a|2k_{1}+3k_{2} \end{eqnarray*}
which is a reformulation of $(\mathrm{T5}^{\prime}).$  The remaining properties can be proved  similarily.
\end{proof}

Define $\mathrm{X}:\mathrm{P} \rtimes \mathbb{N} ^{\times}\to \A$ by $\mathrm{X}_{(m,a)}:=s^{x}t^{y}\upsilon_{a}$ where $m=2x+3y~and ~x, y \in \mathbb{N}.$
If $m=2x+3y=2x^{'}+3y^{'}$ some $x,~y,~x^{'},~y^{'}$~in~$\mathbb{N}$  and  let $x^{'}\geq x,$ then $2(x-x^{'})=3(y^{'}-y)$.  Since
 $  x=x^{'}-3u$  and  $y=y^{'}+2u $
some $u$ in $\mathbb{N}, $  we  have
$$s^{x}t^{y}=s^{x^{'}}s^{*3u}t^{2u}t^{y^{'}}=s^{x^{'}}s^{*3u}s^{3u}t^{y^{'}}=s^{x^{'}}t^{y^{'}}.$$
This shows that the formula  $\mathrm{X}$ is well-defined.

\begin{lemma} \label{le:X} 
The formula $\mathrm{X}_{(m,a)}:=s^{x}t^{y}\upsilon_{a}$ is an isometry representation on $\mathrm{P} \rtimes \mathbb{N} ^{\times}$ into  $\mathcal{A}$~~where $~~m=2x+3y,$ $x \geq 0,$ and $y \geq 0.$
\end{lemma} 
\begin{proof}
For  $(m,a)\in \mathrm{P} \rtimes \mathbb{N} ^{\times},$  we  have
\begin{eqnarray*}
\mathrm{X}_{(m,a)}^{*}\mathrm{X}_{(m,a)} =
  \upsilon_{a}^{*}t^{*y}s^{*x}s^{x}t^{y}\upsilon_{a} =1.
\end{eqnarray*}
    If   we   put $m=2x_{1}+3y_{1}$~and~$n=2x_{2}+3y_{2}$ for $(m,a),~(n,b) \in \mathrm{P} \rtimes \mathbb{N} ^{\times}$ then  we  have
\begin{eqnarray*}
\mathrm{X}_{(m,a)}\mathrm{X}_{(n,b)}&=&(s^{x_{1}}t^{y_{1}}\upsilon_{a})(s^{x_{2}}t^{y_{2}}\upsilon_{b})\\
&=&s^{x_{1}}t^{y_{1}}(\upsilon_{a}s^{x_{2}})t^{y_{2}}\upsilon_{b}\\
&=&s^{x_{1}}t^{y_{1}}s^{ax_{2}}(\upsilon_{a}t^{y_{2}})\upsilon_{b}\\
&=&s^{x_{1}+ax_{2}}t^{y_{1}+ay_{2}}\upsilon_{a}\upsilon_{b}\\
&=&\mathrm{X}_{(m+an,ab)}=\mathrm{X}_{(m,a)(n,b)}.\end{eqnarray*}\end{proof}

\begin{lemma}  \label{le:iso} 
Suppose  that  the representation $\mathrm{X}$ satisfies the Nica-covariance relation $($In fact, this is to be proved later$).$
   If  we  consider a homomorphism $\rho_{\mathrm{W}}$   from  $\mathcal{A}$ into  $\mathcal{T}(\mathrm{P} \rtimes \mathbb{N} ^{\times})$ such that  $\rho_{\mathrm{W}}(s)=\mathrm{W}_{(2,1)},~\rho_{\mathrm{W}}(t)=\mathrm{W}_{(3,1)},$ and~$\rho_{\mathrm{W}}(\upsilon_{p})=\mathrm{W}_{(0,p)}$ for every prime $p$ and $\pi_{s,t,\upsilon}: \mathcal{T}(\mathrm{P} \rtimes \mathbb{N} ^{\times}) \to   \A$ such that  $\pi_{s,t,\upsilon}(\mathrm{W}_{(m,a)})=\mathrm{X}_{(m,a)}$,  then  they  are  the  inverse of each other.
\end{lemma} 
\begin{proof}
  It  is  enough  to  consider  only of   the  form   for $\mathrm{W}_{(m,a)} \mathrm{W}^*_{(n,b)} \in \mathcal{T}(\mathrm{P} \rtimes \mathbb{N} ^{\times})$  by  the  Proposition  \ref{def:uni}.
\begin{eqnarray*}
(\rho_{\mathrm{W}}\cdot \pi)(\mathrm{W}_{(m,a)} \mathrm{W}^{*}_{(n,b)})&=&\rho_{\mathrm{W}}(\mathrm{X}_{(m,a)} \mathrm{X}^{*}_{(n,b)})\\
&=&\mathrm{W}_{(2x+3y,a)}\mathrm{W}^{*}_{(2x'+3y',b)}\\
&=&\mathrm{W}_{(m,a)}\mathrm{W}^{*}_{(n,b)}
\end{eqnarray*}
where $m=2x+3y$ and $n=2x'+3'y  $ some $x, x', y, y'\in \mathbb{N}.$
 And  also  we  have 
\begin{eqnarray*}
(\pi \cdot \rho_{\mathrm{W}})((s^{x}t^{y}v_{a})(s^{x'}t^{y'}v_{b})^{*})&=&\pi(\mathrm{W}^{x}_{(2,1)}\mathrm{W}^{y}_{(3,1)}\mathrm{W}_{(0,a)}\mathrm{W}^{*}_{(0,b)}\mathrm{W}^{*y'}_{(3,1)}\mathrm{W}^{*x'}_{(2,1)})\\
&=&\mathrm{X}^{x}_{(2,1)}\mathrm{X}^{y}_{(3,1)}\mathrm{X}_{(0,a)}\mathrm{X}^{*}_{(0,b)}\mathrm{X}^{*y'}_{(3,1)}\mathrm{X}^{*x'}_{(2,1)}\\
&=&(s^{x}t^{y}v_{a})(s^{x'}t^{y'}v_{b})^{*}.
\end{eqnarray*}
Therefore  we  have  $\pi \cdot \rho_{\mathrm{W}}=i$ and $ \rho_{\mathrm{W}}\cdot \pi=i.$ 
\end{proof}

Lemma \ref{le:X} shows  that $\mathrm{X}$ is an isometric representation of $\mathrm{P} \rtimes \mathbb{N} ^{\times}.$   
Next, we  are  going to prove that the representation  $\mathrm{X}$  satisfies the Nica-covariance relation; the relation is 
$$\mathrm{X}^{*}_{(m,a)}\mathrm{X}_{(n,b)}=\mathrm{X}_{(m,a)^{-1}\sigma}\mathrm{X}^{*}_{(n,b)^{-1}\sigma}$$
where $\sigma=(m,a)\Cup(n,b).$\\

\begin{lemma} 
 Suppose that $(m+a\mathrm{P})\cap(n+b\mathrm{P})\neq\emptyset.$
For $m,n\in\mathrm{P}$ and $a,b\in \mathbb{N}^{\times}$ we let $a^{'}: = a/gcd(a,b),~b^{'}: = b/gcd(a,b),$ and suppose that $(\alpha, ~\beta)$ is the smallest non-negative solution of $(n - m)/gcd(a ,b)=\alpha a^{'} -  \beta b^{'}$ with $\alpha \neq 1,~ \beta \neq 1.$
Then $\mathrm{X}_{(m,a)^{-1}\sigma}\mathrm{X}_{(n,b)^{-1}\sigma}^{*} 
=  \mathrm{X}_{(\alpha,b^{\prime})}\mathrm{X}^{*}_{(\beta,a^{\prime})}$
where  $\sigma =  (m,a)\Cup (n,b).$
 \end{lemma} 

\begin{proof}
By proposition 2.4, $\sigma=(m,a)\Cup(n,b)=(\ell,lcm(a,b))$ 
where $\ell:=m+a\alpha=n+b\beta$ is the smallest element of $(m+a\mathrm{P})\cap(n+b\mathrm{P})$ in   the  usual order.
Let $m=2x_{1}+3y_{1}$  and  $n=2x_{2}+3y_{2}$ some $x_{1},x_{2},y_{1},y_{2}\in\mathbb{N}$, then 
\begin{eqnarray*}
\mathrm{X}^{*}_{(m,a)}\mathrm{X}_{(n,b)}&=&(s^{x_{1}}t^{y_{1}}\upsilon_{a})^{*}(s^{x_{2}}t^{y_{2}}\upsilon_{b})\\
&=&\upsilon_{a}^{*}t^{*y_{1}}s^{*x_{1}}s^{x_{2}}t^{y_{2}}\upsilon_{b}.
\end{eqnarray*}
 Since  
$(m,a)^{-1}\sigma   = (\alpha,   b^{'}) $  and  $ (n,b)^{-1}\sigma =  (\beta,  a^{'}),$
we  have 
\[  \mathrm{X}_{(m,a)^{-1}\sigma}\mathrm{X}_{(n,b)^{-1}\sigma}^{*}=
\begin{cases}
0 &if ~(m +   a\mathrm{P})\cap(n +  b\mathrm{P})=\emptyset,\\
\mathrm{X}_{(\alpha,  b^{\prime})}\mathrm{X}^{*}_{(\beta ,  a^{\prime})} & if~ (m +  a\mathrm{P})\cap(n+b\mathrm{P})\neq\emptyset.
\end{cases} \]

\end{proof}

\begin{lemma} \label{le:0}  
 If $(m+a\mathrm{P})\cap(n+b\mathrm{P})=\emptyset,$ then $\upsilon_{a}^{*}t^{y_{1}*}s^{x_{1}*}s^{x_{2}}t^{y_{2}}\upsilon_{b}=0$ where $m=2x_{1}+3y_{1},~n=2x_{2}+3y_{2}$ some $x_{1},x_{2},y_{1},y_{2}\in \mathbb{N}$
 $($i.e., $\mathrm{X}^{*}_{(m,a)}\mathrm{X}_{(n,b)}=0$$).$
\end{lemma} 

\begin{proof} 
Since $(m+a\mathrm{P})\cap(n+b\mathrm{P})=\emptyset$ ,  $~m-n\not\equiv 0~(\!\!\!\!\mod~gcd(a,b)).$
Then $gcd(a,b)$ has a prime factor $p$ which does not divide $n-m$. We can write $n-m=cp+k$ with $0<k<p.$
Since $c,k \in \mathbb{Z},$ there exist $c_{1},c_{2},k_{1},k_{2}$ in $\mathbb{Z}$ such that $c=2c_{1}+3c_{2}$  and  $ k=2k_{1}+3k_{2}.$
 So   we  have
$$n-m=cp+k=(2c_{1}+3c_{2})p+(2k_{1}+3k_{2})=2(c_{1}p+k_{1})+3(c_{2}p+k_{2})$$
  and  
$$2((x_{2}-x_{1})-(c_{1}p+k_{1}))=3((c_{2}p+k_{2})-(y_{2}-y_{1})).$$
Let $x_{2}-x_{1}=c_{1}p+k_{1}+3u$  and 
  $y_{2}-y_{1}=c_{2}p+k_{2}-2u$
 some  $u$~ in~ $\mathbb{Z}.$
Now we factor $a=a_{0}p,~b=b_{0}p$ and apply $\mathrm{(T4^{'})}$  and  $\mathrm{(T5)}$ to get
\begin{eqnarray*}
\upsilon_{a}^{*}t^{*y_{1}}s^{*x_{1}}s^{x_{2}}t^{y_{2}}\upsilon_{b}&=&\upsilon_{a_{0}}^{*}\upsilon_{p}^{*}s^{((x_{2}-x_{1}))}t^{((y_{2}-y_{1}))}\upsilon_{p}\upsilon_{b_{0}}\\
&=&\upsilon_{a_{0}}^{*}\upsilon_{p}^{*}s^{((c_{1}p+k_{1}+3u))}t^{((c_{2}p+k_{2}-2u))}\upsilon_{p}\upsilon_{b_{0}}\\
&=&\upsilon_{a_{0}}^{*}(\upsilon_{p}^{*}s^{((c_{1}p))})s^{((k_{1}))}s^{((3u))}t^{*((2u))}t^{((k_{2}))}(t^{((c_{2}p))}\upsilon_{p})\upsilon_{b_{0}}.
\end{eqnarray*}
 Since  $s$ and $t^{*}$ are commute and $t^{2}=s^{3}$ by $\mathrm{(R1)}$  and  $\mathrm{(R2)}$,   we  have  $s^{((3u))}t^{*((2u))}=1.$   By $ \mathrm{(T1)}$ 
$\upsilon_{p}^{*}s^{((c_{1}p))}=s^{((c_{1}))}\upsilon_{p}^{*} $   and  $t^{((c_{2}p))}\upsilon_{p}=\upsilon_{p}t^{((c_{2}))}.$
Thus  the  above  equation  can  be  converted  to 
 $$\upsilon_{a}^{*}t^{*y_{1}}s^{*x_{1}}s^{x_{2}}t^{y_{2}}\upsilon_{b}=\upsilon_{a_{0}}^{*}s^{((c_{1}))}(\upsilon_{p}^{*}s^{((k_{1}))}t^{((k_{2}))}\upsilon_{p})t^{((c_{2}))}\upsilon_{b_{0}}.$$ 
By Lemma \ref{le:k1}  we  have  $\upsilon_{p}^{*}s^{((k_{1}))}t^{((k_{2}))}\upsilon_{p}=0$
 for $0< k < p$.
\end{proof}

\begin{lemma} \label{le:not 0}  
Suppose that $(m+a\mathrm{P})\cap(n+b\mathrm{P})\neq\emptyset$ where $m=2x_{1}+3y_{1}$   and $n=2x_{2}+3y_{2}$ some $x_{1},x_{2},y_{1},y_{2}$ in $\mathbb{N}$.  Then $$~\upsilon_{a}^{*}t^{*y_{1}}s^{*x_{1}}s^{x_{2}}t^{y_{2}}\upsilon_{b}=s^{\alpha^{'}}t^{\alpha^{''}}\upsilon_{b^{'}}\upsilon_{a^{'}}^{*}t^{*\beta^{''}}s^{*\beta^{'}}$$ where $a^{'}:=a/gcd(a,b),~b^{'}:=b/gcd(a,b)$ and suppose that $(\alpha, \beta)$ is the smallest non-negative solution of $(n-m)/gcd(a,b)=\alpha a^{'}-\beta b^{'}$ with $\alpha \neq 1,~\beta \neq 1,~\alpha=2\alpha^{'}+3\alpha^{''},$~and~~$\beta=2\beta^{'}+3\beta^{''}$ some $\alpha^{'},\alpha^{''},\beta^{'},\beta^{''}$ in $\mathbb{N}$
 $($i.e., $\mathrm{X}_{(m,a)}^{*}\mathrm{X}_{(n,b)}=\mathrm{X}_{(m,a)^{-1}\sigma}\mathrm{X}_{(n,b)^{-1}\sigma}^{*}$ where $\sigma=(m,a)\Cup(n,b)$$).$
\end{lemma}

\begin{proof}
Since $(m+a\mathrm{P})\cap(n+b\mathrm{P})\neq\emptyset$ implies $m\equiv n(\!\!\!\mod gcd(a,b)),$
  we  put  $k=(n-m)/gcd(a,b).$  First suppose that $k>0.$
Let $\mathrm{G}=gcd(a,b),$ then $(n-m)=k\mathrm{G}$.  We  put  $a=a^{'}\mathrm{G}$ and $b=b^{'}\mathrm{G }$   where $ gcd(a',b')=1.$ Since  $k\mathrm{G}=n-m =2(x_{2}-x_{1})+3(y_{2}-y_{1}),$
we  take $k_{1},k_{2}$ in $\mathbb{Z}$ such that $k=2k_{1}+3k_{2}$.  Since  $2(k_{1}\mathrm{G}-x_{2}+x_{1})=3(y_{2}-y_{1}-k_{2}\mathrm{G}),$
    we  have  $k_{1}\mathrm{G}-x_{2}+x_{1}=3u$   and  $y_{2}-y_{1}-k_{2}\mathrm{G}=2u~~~ $ some $u\in\mathbb{Z}.$
  From  the    equation 
\[  \begin{cases}
x_{2}-x_{1}=k_{1}\mathrm{G}-3u,\\
y_{2}-y_{1}=k_{2}\mathrm{G}+2u, \end{cases} \] 
  we have 
\begin{eqnarray*}
\upsilon_{a}^{*}t^{*y_{1}}s^{*x_{1}}s^{x_{2}}t^{y_{2}}\upsilon_{b}&=&\upsilon_{a}^{*}s^{((x_{2}-x_{1}))}t^{((y_{2}-y_{1}))}\upsilon_{b}\\
&=&\upsilon_{a}^{*}s^{((k_{1}\mathrm{G}-3u))}t^{((k_{2}\mathrm{G}+2u))}\upsilon_{b}\\
&=&\upsilon_{a}^{*}s^{((k_{1}\mathrm{G}))}s^{*((3u))}t^{((2u))}t^{((k_{2}\mathrm{G}))}\upsilon_{b}\\
&=&\upsilon_{a^{'}}^{*}\upsilon_{\mathrm{G}}^{*}s^{((k_{1}\mathrm{G}))}t^{((k_{2}\mathrm{G}))}\upsilon_{\mathrm{G}}\upsilon_{b^{'}}\\
&=&\upsilon_{a^{'}}^{*}\upsilon_{\mathrm{G}}^{*}(s^{((k_{1}))}t^{((k_{2}))})^{\mathrm{G}}\upsilon_{\mathrm{G}}\upsilon_{b^{'}}\\
&=&\upsilon_{a^{'}}^{*}s^{((k_{1}))}t^{((k_{2}))}\upsilon_{b^{'}}.
\end{eqnarray*}
By Lemma \ref{le:k1},  it  is   enough  to  consider a particular solution of $k=2k_{1}+3k_{2}$.
If $k=1$,   then  we can take $k_{1} = -1$  and   $k_{2} = 1$.  Since $1=\alpha a^{'} -  \beta b^{'}$ and  $s^{3}=t^{2}$,  we  have
\begin{eqnarray*}
\upsilon_{a^{'}}^{*}s^{((k_{1}))}t^{((k_{2}))}\upsilon_{b^{'}}
&=&\upsilon_{a^{'}}^{*}s^{*}t\upsilon_{b^{'}}\\
&=&\upsilon_{a^{'}}^{*}s^{*(\alpha a^{'}-\beta b^{'})}t^{\alpha a^{'}-\beta b^{'}}\upsilon_{b^{'}}\\
&=&s^{*\alpha}t^{\alpha}\upsilon^{*}_{a^{'}}\upsilon_{b^{'}}s^{\beta}t^{*\beta}\\
&=&s^{\alpha^{'}}t^{\alpha''}\upsilon^{*}_{a^{'}}\upsilon_{b^{'}}t^{*\beta''}s^{*\beta^{'}}.
\end{eqnarray*}
If $k>1$, we can take $k_{1}\geq 0$  and  $ k_{2}\geq 0$ such that $2k_{1}+3k_{2}=k$.  Since   $~t^{2}=s^{3},$  we  can  get  
\[ t^{k_{2}}= \begin{cases}
s^{3k^{'}}~(k^{'}~\in~\{1,2,3,\cdots\})&   \text {if} ~k_{2}~ \text {is~even},\\
ts^{3k^{'}}~(k^{'}~\in~\{0,1,2,\cdots\})&  \text { if}~k_{2}~ \text {is~odd}.
\end{cases} \] 
Suppose that $k_{2}$ is even. Let $t^{k_{2}}=s^{3k^{'}}$ some $k^{'}\geq 1.$ Then  we  have 
 \begin{eqnarray*}
\upsilon_{a^{'}}^{*}(s^{k_{1}}t^{k_{2}})\upsilon_{b^{'}}
=\upsilon_{a^{'}}^{*}s^{k_{1}}s^{3k^{'}}\upsilon_{b^{'}}. 
\end{eqnarray*}
Peeling one factor off $s^{k_{1}}$ and applying the adjoint of $\mathrm{(T4)}$ gives $s^{k_{1}}=ss^{k_{1}-1}$ and
$\upsilon_{a^{'}}^{*}s=s\upsilon_{a^{'}}^{*}s^{*(a^{'}-1)}, $   hence  we  have 
\begin{eqnarray*}
&&\upsilon_{a^{'}}^{*}s^{k_{1}}s^{3k^{'}}\upsilon_{b^{'}}\\
&&=(\upsilon_{a^{'}}^{*}s)s^{k_{1}-1}s^{3k^{'}}\upsilon_{b^{'}}\\
&&=s^{\mathbf{1}}\upsilon_{a^{'}}^{*}s^{k_{1}-\mathbf{1}a^{'}}s^{3k^{'}}\upsilon_{b^{'}}\\
&&=s(\upsilon_{a^{'}}^{*}s)s^{k_{1}-a^{'}-1}s^{3k^{'}}\upsilon_{b^{'}}\\
&&=s^{\mathbf{2}}\upsilon_{a^{'}}^{*}s^{k_{1}-\mathbf{2}a^{'}}s^{3k^{'}}\upsilon_{b^{'}}.
\end{eqnarray*}
If $k_{1}-2a^{'}>0,$ we  peel another $s$ off $s^{k_{1}-2a^{'}}$ and pull it across $\upsilon^{*}_{a^{'}}.$ 
 We repeat this  process if $k_{1}-3a^{'}>0.$ The number of times  which  we can do this is precisely the number $\alpha_{0}$ appearing in the Euclidean algorithm, applied to $a^{'},~b^{'}$ and $k_{1}$.
Continuing until $-a^{'}<k_{1}-\alpha_{0}a^{'}\leq 0,$ we have 
$$\upsilon_{a^{'}}^{*}s^{k_{1}+3k^{'}}\upsilon_{b^{'}}=s^{\alpha_{0}}\upsilon_{a^{'}}^{*}s^{3k^{'}}s^{*(\alpha_{0}a^{'}-k_{1})}\upsilon_{b^{'}}.$$
Now we apply $(\mathrm{T4})$  to  $s^{*}\upsilon_{b^{'}}=s^{(b^{'}-1)}\upsilon_{b^{'}}s^{*}$  so  as to pull factors $s^{*}$ through $\upsilon_{b^{'}}:$
\begin{eqnarray*}
&&s^{*(\alpha_{0}a^{'}-k_{1})}\upsilon_{b^{'}}\\
&&=s^{*(\alpha_{0}a^{'}-k_{1}-1)}(s^{(b^{'}-1)}\upsilon_{b^{'}}s^{*})\\
&&=s^{*(\alpha_{0}a^{'}-k_{1}-\mathbf{1}b^{'})}\upsilon_{b^{'}}s^{*\mathbf{1}}\\
&&=s^{*(\alpha_{0}a^{'}-k_{1}-1-b^{'})}(s^{(b^{'}-1)}\upsilon_{b^{'}}s^{*})s^{*}\\
&&=s^{*(\alpha_{0}a^{'}-k_{1}-\mathbf{2}b^{'})}\upsilon_{b^{'}}s^{*\mathbf{2}}.
\end{eqnarray*}
We can repeat this   process until $b^{'}>k_{1}-\alpha_{0}a^{'}+\beta_{0}b^{'}\geq 0$, then  we  have
$$s^{*(\alpha_{0}a^{'}-k_{1})}\upsilon_{b^{'}}=s^{(k_{1}-\alpha_{0} a^{'} +  
\beta_{0} b^{'})}\upsilon_{b^{'}}s^{*\beta_{0}}.$$ 
Thus it  leads  to us  
$$\upsilon_{a^{'}}^{*}(s^{k_{1}}t^{k_{2}})\upsilon_{b^{'}}=s^{\alpha_{0}}(\upsilon_{a^{'}}^{*}s^{3k^{'}}s^{k_{1}-\alpha_{0}a^{'}+\beta_{0}b^{'}}\upsilon_{b^{'}})s^{*\beta_{0}}.$$
If we cosider $k_{1}-\alpha_{0}a^{'}+\beta_{0}b^{'}$ like the first $k_{1}$ and repeat the  same  process, we can choose $\alpha_{1}$  and   $\beta_{1}$ such that $-a^{'} <(k_{1}-\alpha_{0}a^{'} +  \beta_{0}  b^{'} )- \alpha_{1} a^{'} \leq 0 $
 and $0 \leq (k_{1}-\alpha_{0}a^{'}  +   \beta_{0} b^{'})  - \alpha_{1}a^{'}+ \beta_{1}b^{'}<b^{'}$. Then we have
 $$\upsilon_{a^{'}}^{*}(s^{k_{1}}t^{k_{2}})\upsilon_{b^{'}}=s^{\alpha_{0}}s^{\alpha_{1}}(\upsilon_{a^{'}}^{*}s^{3k^{'}}s^{k_{1}-\alpha_{0}a^{'}+\beta_{0}b^{'}}\upsilon_{b^{'}})s^{*\beta_{1}}s^{*\beta_0}.$$
  And given $\alpha_{i}$ for $0 \leq i <n$ and  $\beta_i$ for $0 \leq i <n,$
  define  $\beta_{n}$ by  $0  \leq k_{1}-(\sum _{i=0}^{n}\alpha_{i}) a^{'}+(\sum _{i=0}^{n}\beta_{i}) b^{'}< b^{'}$
and  $\alpha_{n+1}$ by  $-a^{'} < k_{1}-(\sum _{i=0}^{n+1}\alpha_{i}) a^{'}+(\sum _{i=0}^{n}\beta_{i}) b^{'}\leq 0$. By Euclidean algorithm, there exist $n(\alpha^{'}) $ and $ n (\beta^{'})$  such that $\alpha_{i}=0$  for $i>n(\alpha^{'}) $ and $\beta_{i}=0$  for $i>n(  \beta^{'}). $ Then the pair $$(\alpha^{'},\beta^{'})=(\sum_{i=0}^{n(\alpha^{'})}\alpha_{i},\sum_{i=0}^{n(\beta^{'})}\beta_{i})$$
is the non-negative solution of $k_{1}=\alpha^{'}a^{'}-\beta^{'}b^{'}$.    We have
\begin{eqnarray*}
\upsilon_{a^{'}}^{*}(s^{k_{1}}t^{k_{2}})\upsilon_{b^{'}}&=&s^{\alpha_{0}}s^{\alpha_{1}}\cdots s^{\alpha_{n(\alpha^{'})}}(\upsilon_{a^{'}}^{*}s^{3k^{'}}s^{k_{1}-\alpha^{'}a^{'}  +  \beta^{'}b^{'}}\upsilon_{b^{'}})s^{*\beta_{0}}s^{*\beta_{1}}\cdots s^{*\beta_{n(  \beta^{'})}}\\
&=&s^{\alpha^{'}}(\upsilon_{a^{'}}^{*}s^{3k^{'}}s^{k_{1}-\alpha^{'}a^{'}+\beta^{'}b^{'}}\upsilon_{b^{'}})s^{*\beta^{'}}\\
&=&s^{\alpha^{'}}\upsilon_{a^{'}}^{*}t^{k_{2}}\upsilon_{b^{'}}s^{*\beta^{'}}.
\end{eqnarray*}
Remembering that $k=\alpha a^{'}-\beta b^{'},~k=2k_{1}+3k_{2}$,   and  $(\alpha^{'},\beta^{'})$ is the non-negative solution of $k_{1}=\alpha^{'} a^{'}-\beta^{'} b^{'},$
we have \begin{eqnarray*}
3k_{2}&=&\alpha a^{'}-\beta b^{'}-2\alpha^{'} a^{'}+2\beta^{'} b^{'}\\
&=&a^{'}(\alpha-2\alpha^{'})-b^{'}(\beta-2\beta^{'}).\end{eqnarray*}
 Here, we can take $(\alpha-2\alpha^{'})=3\alpha^{''}$  and  $(\beta-2\beta^{'})=3\beta^{''}$ because $gcd(a^{'}, b^{'})=1$.  Then  $k_{2}=\alpha^{''}a^{'}-\beta^{''}b^{'}.$\\
Since $k_{2}=\alpha^{''}a^{'}-\beta^{''}b^{'},$  we  have  
\begin{eqnarray*}
&&s^{\alpha^{'}}(\upsilon_{a^{'}}^{*}t^{\alpha^{''}a^{'}})(t^{*\beta^{''}b^{'}}\upsilon_{b^{'}})s^{*\beta^{'}}\\
&&=s^{\alpha^{'}}(t^{\alpha^{''}}\upsilon_{a^{'}}^{*})(\upsilon_{b^{'}}t^{*\beta^{''}})s^{*\beta^{'}}\\
&&=s^{\alpha^{'}}t^{\alpha^{''}}\upsilon_{b^{'}}\upsilon_{a^{'}}^{*}t^{*\beta^{''}}s^{*\beta^{'}}.
\end{eqnarray*}
It    is  sufficient  to compute for $k<0.$
If $k<0$  and $k\neq-1$,  then we can take $k_{1}\leq0$ and $k_{2}\leq0.$  
  It  leads  to  us  that 
\begin{eqnarray*}
&&\upsilon_{a^{'}}^{*}s^{*|k_{1}|}t^{*|k_{2}|}\upsilon_{b^{'}}\\
&&=(\upsilon_{b^{'}}^{*}t^{|k_{2}|}s^{|k_{1}|}\upsilon_{a^{'}})^{*}\\
&&=(s^{\beta^{'}}t^{\beta^{''}}\upsilon_{a^{'}}\upsilon_{b^{'}}^{*}t^{*\alpha^{''}}s^{*\alpha^{'}})^{*}\\
&&=s^{\alpha^{'}}t^{\alpha^{''}}\upsilon_{b^{'}}\upsilon_{a^{'}}^{*}t^{*\beta^{''}}s^{*\beta^{'}}
\end{eqnarray*}
where $(\alpha^{'},\beta^{'})$ is the smallest non-negative solution of $k_{1}=\alpha^{'}a^{'}-\beta^{'} b^{'}.$ 
 When $k=-1,$ we can get a result by similar computation  of  the  case $k=1.$ 

Now suppose that $k_{2}$ is odd. Let  $t^{k_{2}}=ts^{3k^{'}}$ some $k^{'} \geq 0.$  Then  we  have 
 \begin{eqnarray*}\upsilon_{a^{'}}^{*}(s^{k_{1}}t^{k_{2}})\upsilon_{b^{'}}=\upsilon_{a^{'}}^{*}s^{k_{1}}ts^{3k^{'}}\upsilon_{b^{'}}=\upsilon_{a^{'}}^{*}ts^{k_{1}+3k^{'}}\upsilon_{b^{'}}.\end{eqnarray*}
To compute odd case, we can prove by moving the position of t appropriately because $t$ and $s$ are commute.
\end{proof}

\begin{theorem} \label{th:den}  
Let $\mathcal{A}$ be the universal $C^{*}$-algebra generated by isometries $s,t$, and $\{\upsilon_{p}:p~is~prime\}$ satisfying relations $\mathrm{(R1,R2)}$  and $ \mathrm{(T1- T6)}.$ Then there is an isomorphism $\pi$ of  $\mathcal{T}(\mathrm{P} \rtimes \mathbb{N} ^{\times})$ onto $\mathcal{A}$ such that $\pi(\mathrm{W}_{(2,1)})=s,$ $~\pi(\mathrm{W}_{(3,1)})=t,$ and $\pi(\mathrm{W}_{(0,p)})=\upsilon_{p}$  for every prime $p$. 
\end{theorem}

\begin{proof}
By Lemma \ref{le:0}  and  \ref{le:not 0}, we  can  see   that  the formular $\mathrm{X}_{(m,a)}:=s^{x}t^{y}\upsilon_{a}$ where $m=2x+3y$ some $x,y$ in $\mathbb{N}$ defines  Nica-covariant isometric representation
 $\mathrm{X}=\mathrm{X}_{s,t,\upsilon}$ on $\mathrm{P} \rtimes \mathbb{N} ^{\times}$ 
 into  $\mathcal{A}.$     Since  $\mathcal{T}(\mathrm{P} \rtimes \mathbb{N} ^{\times})$  is  the universal $C^{*}$-algebra for covariant isometric representations of $\mathrm{P} \rtimes \mathbb{N} ^{\times}$,  it   induces a homomorphism $\pi_{s,t,\upsilon}: \mathcal{T}(\mathrm{P} \rtimes \mathbb{N} ^{\times}) \to \mathcal{A}$  such that   $\pi_{s,t,\upsilon}(\mathrm{W}_{(m,a)})=\mathrm{X}_{(m,a)}. $   By Lemma  \ref{le:iso}, $ \pi_{s,t,\upsilon}  $ is an isomorphism of $\mathcal{T}(\mathrm{P} \rtimes \mathbb{N} ^{\times})$ onto $\mathcal{A}.  $\end{proof}
 
    Moreover,  since  ${span}\{\mathrm{W}_{(m,a)}\mathrm{W}_{(n,b)}^{*}:(m,a),(n,b)\in \mathrm{P} \rtimes \mathbb{N} ^{\times}\}$ is a dense $*$-subalgebra of 
$\mathcal{T}(\mathrm{P} \rtimes \mathbb{N} ^{\times}),$   we  have   also  that  ${span} \{s^{x_{1}}t^{y_{1}}\upsilon_{a}\upsilon_{b}^{*}t^{*y_{2}}s^{*x_{2}}:(m,a),(n,b)\in \mathrm{P} \rtimes \mathbb{N} ^{\times}$,  $m=2x_{1}+3y_{1},~ n=2x_{2}+3y_{2}$, and $x_{1},x_{2},y_{1},y_{2} \in \mathbb{N}\}$ is a dense $*$-subalgebra of $\mathcal{A}.$


\section{KMS states  on a   generalized  Toeplitz  Algebra }
\label{intro}

 If  we  consider   the unitary representation $u : \mathbb{R} \to  \mathcal U (\ell^{2}(\mathrm{P} \rtimes \mathbb{N} ^{\times}))$ defined  by
$$u_{r}e_{(m,a)}:=a^{ir}e_{(m,a)}$$ 
where   $\{  e_{(m,a)}  :   (m,a)  \in   \mathrm{P} \rtimes \mathbb{N}  \}$   is  the  canonical  orthonormal  basis  of  $\ell^{2}(\mathrm{P} \rtimes \mathbb{N} ^{\times}) $
  and   $ \mathcal U (\ell^{2}(\mathrm{P} \rtimes \mathbb{N}^{\times}))$   is  the   group  of   unitary  operators  in $\mathcal B  (\ell^{2}(\mathrm{P} \rtimes \mathbb{N} ^{\times})),$ 
 then $\{  u_r  |  r  \in  \mathbb R \}$ induces   the  automorphism  group 
 $\tau_r(a) = u_r a u_r^* (  a \in   {\mathcal{C}}_{red} (\mathrm{P} \rtimes \mathbb{N} ^{\times})$) of  the  reduced  semigroup $C^*$-algebra  ${\mathcal{C}}_{red} (\mathrm{P} \rtimes \mathbb{N}^{\times})$  on $\ell^{2}(\mathrm{P} \rtimes \mathbb{N}^{\times})$.
  In  fact, the   definition  of  the  left  regular  isometric   representation  gives   the  following  equations 
  $$\tau_{r}( \W_{(2,1)})= \W_{(2,1 )},~\tau_{r}(  { \W_{(3,1)} }  )= { \W_{(3,1) }},  ~\text {and} ~\tau_{r}(\W_{(0,p) })=p^{ir}\W_{(0,p) } $$  
for  prime  $ p  $ and  $r\in\mathbb{R}$   where $\W: \mathrm{P} \rtimes \mathbb{N}^{\times}  \to  \mathcal B  (\ell^{2}(\mathrm{P} \rtimes \mathbb{N} ^{\times}))$  is
  the  left  regular  representation  on $ \mathrm{P} \rtimes \mathbb{N}^{\times}.$
By the  universality of  $  \mathcal{T}(\mathrm{P} \rtimes \mathbb{N}^{\times})$  there  is  a  $*$-homomorphism  $\Phi $ from
 $  \mathcal{T}(\mathrm{P} \rtimes \mathbb{N}^{\times})$ onto $ {\mathcal{C}}_{red}(\mathrm{P} \rtimes \mathbb{N} ^{\times})$ 
where $\Phi ( W_{(m,a)}  )=  \W_{(m, a)}$  for  $(m, a)  \in \mathrm{P} \rtimes \mathbb{N}^{\times}.$
 Thus we  can  see  that there is a strongly continuous action $\sigma$ of $\mathbb{R}$ on $\mathcal{T}(\mathrm{P} \rtimes \mathbb{N} ^{\times})$   such that 
$$\sigma_{r}(s)=s,~\sigma_{r}(t)=t~\text{and}~\sigma_{r}(\upsilon_{p})=p^{ir}\upsilon_{p}~\text{for~ prime}~p~\text{and}~r \in\mathbb{R}.$$

Suppose that $\alpha$ is an action of $\mathbb{R}$ on a $C^{*}$-algebra $\mathcal{B}.$ An element $a$ of $\mathcal{B}$ is analytic for the action $\alpha$ if the function $r \mapsto \alpha_{r}(a)$ is the restriction to $\mathbb{R}$ of an entire function on $\mathbb{C};$
the set $\mathcal{B}^{a}$ of analytic element is a dense $*$-subalgebra of $\mathcal{B}.$

\begin{proposition} \label{anal}
For our system  $(\mathcal{T}(\mathrm{P} \rtimes \mathbb{N} ^{\times}),\mathbb{R},\sigma)$  the elements $s^{x_{1}}t^{y_{1}}\upsilon_{a}\upsilon_{b}^{*}t^{*y_{2}}s^{*x_{2}}$ for $\mathcal{T}(\mathrm{P} \rtimes \mathbb{N} ^{\times})$ are all analytic for $x_{1},y_{1},x_{2},y_{2}\in \mathbb{N}$ and $a,~b \in\mathbb{N} ^{\times}.$
\end{proposition}
\begin{proof}
By   the  definition of $\sigma_{r}$ 
\begin{eqnarray*}
\sigma_{r}(s^{x_{1}}t^{y_{1}}\upsilon_{a}\upsilon_{b}^{*}t^{*y_{2}}s^{*x_{2}})
&=&(a^{ir}s^{x_{1}}t^{y_{1}}\upsilon_{a})(b^{-ir}\upsilon_{b}^{*}t^{*y_{2}}s^{*x_{2}})\\
&=&(ab^{-1})^{ir}(s^{x_{1}}t^{y_{1}}\upsilon_{a}\upsilon_{b}^{*}t^{*y_{2}}s^{*x_{2}}).\end{eqnarray*}
Therefore the function $r \mapsto \sigma_{r}(s^{x_{1}}t^{y_{1}}\upsilon_{a}\upsilon_{b}^{*}t^{*y_{2}}s^{*x_{2}})$ is the restriction to $\mathbb{R}$ of  an entire function on $\mathbb{C}.$  Moreover, $s^{x}t^{y}\upsilon_{a}$  and $\upsilon_{b}^{*}t^{*y'}s^{*x'}$ are all analytic  for $x,y,x',y'\in \mathbb{N}$ and $a,~b \in\mathbb{N} ^{\times}.$ \end{proof}

Now we introduce a KMS-state for  a  $C^*$-dynamical  system.  Let  $\mathcal B$ be  a $C^*$-algebra    and $\sigma$  be  an  action  of $\mathbb R$  on  $\mathcal B$.
For $\beta\in(0,\infty)$ a state $\phi$ of $\mathcal{B}$ is  {\it  a $\mathrm{KMS}$ state at inverse temperature $\beta$ for $\sigma,$ or a $\mathrm{KMS}_{\beta}$ state for $\sigma,$} if it satisfies the following   {\it  $\mathrm{KMS}_{\beta}$ condition} ;
$$\phi(ab)
 =\phi(b \sigma_{i\beta}(  a) )~~~~\text{for}~~~a,  b   \in \mathcal{B}^{a}$$
 where $\mathcal{B}^{a}$ is the set of analytic element.

\begin{theorem} 
The system $(\mathcal{T}(\mathrm{P} \rtimes \mathbb{N} ^{\times}),\mathbb{R},\sigma)$ has no $\mathrm{KMS}_{\beta}$ state for $\beta<1.$ 
\end{theorem}

\begin{proof}
Suppose   that  $\psi$ is a $\mathrm{KMS}_{\beta}$ state for $\sigma.$
 The $\mathrm{KMS}_{\beta}$ condition implies that  for $a \in \mathbb{N}^{\times}$ and  $0 \leq k<a$ where $k=2x+3y$ some $x,y$ in $\mathbb{N}$(fix $a$) 
  it  holds
\begin{eqnarray*}&&\psi((s^{x}t^{y}\upsilon_{a})(\upsilon_{a}^{*}t^{*y}s^{*x}))\\
&&=\psi((\upsilon_{a}^{*}t^{*y}s^{*x})\sigma_{i\beta}(s^{x}t^{y}\upsilon_{a}))\\
&&=a^{-\beta}\psi(\upsilon_{a}^{*}t^{*y}s^{*x}s^{x}t^{y}\upsilon_{a})\\
&&=a^{-\beta}\psi(1)\\
&&=a^{-\beta}.\end{eqnarray*}
The relation $\mathrm{(T5')}$ and Lemma \ref{le:k1} implies that $\upsilon_{a}^{*}s^{((x))}t^{((y))}\upsilon_{a}=0$ for $0\leq 2x+3y<a$.   And  the projections $s^{x}t^{y}\upsilon_{a}\upsilon_{a}^{*}t^{*y}s^{*x}$ for $0
\leq k<a$ are mutually orthogonal  because  
\begin{eqnarray*}
(s^{x}t^{y}\upsilon_{a}\upsilon_{a}^{*}t^{*y}s^{*x})(s^{x}t^{y}\upsilon_{a}\upsilon_{a}^{*}t^{*y}s^{*x})
&=&s^{x}t^{y}\upsilon_{a}(\upsilon_{a}^{*}t^{*y}s^{*x}s^{x}t^{y}\upsilon_{a})\upsilon_{a}^{*}t^{*y}s^{*x}\\
&=&s^{x}t^{y}\upsilon_{a}\upsilon_{a}^{*}t^{*y}s^{*x}
\end{eqnarray*}
and 
\begin{eqnarray*}
(s^{x}t^{y}\upsilon_{a}\upsilon_{a}^{*}t^{*y}s^{*x})(s^{x^{'}}t^{y^{'}}\upsilon_{a}\upsilon_{a}^{*}t^{*y^{'}}s^{*x^{'}})
&=&s^{x}t^{y}\upsilon_{a}(\upsilon_{a}^{*}t^{((y^{'}-y))}s^{((x^{'}-x))}\upsilon_{a})\upsilon_{a}^{*}t^{*y^{'}}s^{*x^{'}}=0 
\end{eqnarray*}
for $0\leq k^{'}<a$ and $k^{'}=2x^{'}+3y^{'}$ some $x^{'},y^{'}$ in $\mathbb{N}. $   Actually we can take $0< k-k^{'}<a$ (or $0< k^{'}-k<a$)  without  loss  of  generality   because  of  $\mathrm{(T1)},$   and  we  get
 $$1\geq \sum_{k=0}^{a-1}s^{x}t^{y}\upsilon_{a}\upsilon_{a}^{*}t^{*y}s^{*x}.$$
  Since  $\psi$   is  positive,  it implies that 
$$1=\psi(1)\geq\psi(\sum_{k=0}^{a-1}s^{x}t^{y}\upsilon_{a}\upsilon_{a}^{*}t^{*y}s^{*x})=aa^{-\beta},$$
which implies $\beta\geq  1.$ \end{proof}

\begin{theorem} \label{4.3}
Let  ~$\beta\in [1,\infty)$     and  a state $\phi$ of $\mathcal{T}(\mathrm{P} \rtimes \mathbb{N} ^{\times})$   be  a 
$\mathrm{KMS}_{\beta}$ state for $\sigma$.  Then for every  $a,b \in \mathbb{N} ^{\times}$ and $m,n  \in \mathrm{P}$ where $m=2x_{1}+3y_{1},~n=2x_{2}+3y_{2}$ some  $x_{1},x_{2},y_{1},y_{2}$ in $\mathbb{N},$ and some $u \in \mathbb{Z}$  we have 
\[  \phi(s^{x_{1}}t^{y_{1}}\upsilon_{a}\upsilon_{b}^{*}t^{*y_{2}}s^{*x_{2}})=
\begin{cases}
0 &if~a\neq b~~or~~m\not \equiv n~(\!\!\!\!\!\mod~a), \\
a^{-\beta}\phi(t^{((\frac{y_{1}-y_{2}+2u}{a}))}s^{((\frac{x_{1}-x_{2}-3u}{a}))})&if~a=b~~and~~m \equiv n~(\!\!\!\!\!\mod~a).
\end{cases} \]

\end{theorem}

\begin{proof} Suppose that $\phi$ is a $\mathrm{KMS}_{\beta}$ state. Applying the $\mathrm{KMS}$ condition twice gives 
\begin{eqnarray*}
&&\phi(s^{x_{1}}t^{y_{1}}\upsilon_{a}\upsilon_{b}^{*}t^{*y_{2}}s^{*x_{2}})\\
&&=\phi((s^{x_{2}}t^{y_{2}}\upsilon_{b})^{*}\sigma_{i\beta}(s^{x_{1}}t^{y_{1}}\upsilon_{a}))\\
&&=a^{-\beta}\phi((s^{x_{2}}t^{y_{2}}\upsilon_{b})^{*}(s^{x_{1}}t^{y_{1}}\upsilon_{a}))\\
&&=a^{-\beta}\phi((s^{x_{1}}t^{y_{1}}\upsilon_{a})(\sigma_{i\beta}(s^{x_{2}}t^{y_{2}}\upsilon_{b}))^{*})\\
&&=a^{-\beta}\phi((s^{x_{1}}t^{y_{1}}\upsilon_{a})b^{\beta}(s^{x_{2}}t^{y_{2}}\upsilon_{b})^{*})\\
&&=(\frac{a}{b})^{-\beta}\phi(s^{x_{1}}t^{y_{1}}\upsilon_{a}\upsilon_{b}^{*}t^{*y_{2}}s^{*x_{2}}).
\end{eqnarray*}
  And  this  implies that 
\[  \phi(s^{x_{1}}t^{y_{1}}\upsilon_{a}\upsilon_{b}^{*}t^{*y_{2}}s^{*x_{2}})=
\begin{cases}
0 &if~a\neq b,\\
a^{-\beta}\phi(\upsilon_{b}^{*}t^{((y_{1}-y_{2}))}s^{((x_{1}-x_{2}))}\upsilon_{a})&if~a=b.\end{cases} \]

Suppose first that $a=b$ and $m\not \equiv n~(\!\!\!\!\mod~a).$    
  If  $1\leq m-n<a,$  then $\upsilon_{a}^{*}t^{((y_{1}-y_{2}))}s^{((x_{1}-x_{2}))}\upsilon_{a}=0$  by   $\mathrm{(T5')}$ and Lemma \ref{le:k1}. If $m-n>a,$ we can take $m-n=aq+r$ $(q \in \mathbb{N}$ and $1\leq r <a )$ where $q=2q_{1}+3q_{2}$ and  $r=2r_{1}+3r_{2}.$ Since $2(x_{1}-x_{2})+3(y_{1}-y_{2})=a(2q_{1}+3q_{2})+(2r_{1}+3r_{2}),$ we have  $x_{1}-x_{2}=3u_{0}+aq_{1}+r_{1}$ and  $y_{1}-y_{2}=aq_{2}+r_{2}-2u_{0}$ some $u_{0} \in \mathbb{Z}.$ Thus
 \begin{eqnarray*}
\upsilon_{a}^{*}t^{((y_{1}-y_{2}))}s^{((x_{1}-x_{2}))}\upsilon_{a}&=&(\upsilon_{a}^{*}t^{((aq_{2}))})t^{((r_{2}))}(t^{*((2u_{0}))}s^{((3u_{0}))})s^{((r_{1}))}(s^{((aq_{1}))}\upsilon_{a})\\
&=&t^{((q_{2}))}(\upsilon_{a}^{*}t^{((r_{2}))}s^{((r_{1}))}\upsilon_{a})s^{((q_{1}))}\\
&=&0
\end{eqnarray*} 
 because $1 \leq r<a.$  It  follows  that 
\begin{eqnarray*}\phi(s^{x_{1}}t^{y_{1}}\upsilon_{a}\upsilon_{a}^{*}t^{*y_{2}}s^{*x_{2}})&=&a^{-\beta}\phi(\upsilon_{a}^{*}t^{((y_{1}-y_{2}))}s^{((x_{1}-x_{2}))}\upsilon_{a})=a^{-\beta}\phi(0)=0. 
\end{eqnarray*}

Suppose that $a=b$ and $m \equiv n~(\!\!\mod~a).$  Since $m \equiv n~(\!\!\!\mod~a)$ if and only if there is $u \in \mathbb{Z}$ such that $x_{1}-x_{2}\equiv3u~(\!\!\!\mod~a)$ and $y_{2}-y_{1}\equiv2u~(\!\!\!\mod~a),$  we can take $x_{1}-x_{2}=3u+ak$ and $y_{1}-y_{2}=-2u+ak'$ some $k, k' \in \mathbb{Z}.$ By $(\mathrm{T1})$
\begin{eqnarray*}
\upsilon_{a}^{*}t^{((y_{1}-y_{2}))}s^{((x_{1}-x_{2}))}\upsilon_{a}&=&\upsilon_{a}^{*}t^{((ak'))}t^{*((2u))}s^{((3u))}s^{((ak))}\upsilon_{a}\\
&=&t^{((k'))}\upsilon_{a}^{*}\upsilon_{a}s^{((k))}\\
&=&t^{((k'))}s^{((k))}\\
&=&t^{((\frac{y_{1}-y_{2}+2u}{a}))}s^{((\frac{x_{1}-x_{2}-3u}{a}))}.
\end{eqnarray*}
Therefore   we   have   $a^{-\beta}\phi(\upsilon_{a}^{*}t^{((y_{1}-y_{2}))}s^{((x_{1}-x_{2}))}\upsilon_{a})=a^{-\beta}\phi(t^{((\frac{y_{1}-y_{2}+2u}{a}))}s^{((\frac{x_{1}-x_{2}-3u}{a}))}).$\end{proof}

\begin{theorem}  \label{4.4}
Let  $\beta   \in   [1,\infty) $.  If    a   state $\phi$ of $\mathcal{T}(\mathrm{P} \rtimes \mathbb{N} ^{\times})$  satisfies 
\begin{equation} \label{eq:KMS}
\phi(s^{x_{1}}t^{y_{1}}\upsilon_{a}\upsilon_{b}^{*}t^{*y_{2}}s^{*x_{2}})=
\begin{cases}
0 &if~a\neq b ~or~m\not\equiv n~(\!\!\!\!\!\mod~a),\\
a^{-\beta}\phi(t^{((\frac{y_{1}-y_{2}+2u}{a}))}s^{((\frac{x_{1}-x_{2}-3u}{a}))})&if~a=b ~and~m\equiv n~(\!\!\!\!\!\mod~a)\\\end{cases}
\end{equation}
for $a,b \in \mathbb{N} ^{\times}$ and $m,n  \in \mathrm{P}$ where $m=2x_{1}+3y_{1},~n=2x_{2}+3y_{2}$ some  $x_{1},x_{2},y_{1},y_{2}$ in $\mathbb{N},$ and some $u \in \mathbb{Z},$
then $\phi$ is a  $\mathrm{KMS}_{\beta}$ state for $\sigma.$
\end{theorem}

\begin{proof}
Suppose that  $\phi$ satisfies (\ref{eq:KMS}). Since it suffices to check the $\mathrm{KMS}$ condition  holds  on   dense spanning elements, $\phi$  is a  $\mathrm{KMS}_{\beta}$ state for $\sigma$ if and only if 
\begin{eqnarray*}
\phi(xy)=\phi(y\sigma_{i\beta}(x))=\phi(y(\frac{a}{b})^{-\beta}x)=(\frac{a}{b})^{-\beta}\phi(yx)
\end{eqnarray*}
where $x=s^{x_{1}}t^{y_{1}}\upsilon_{a}\upsilon_{b}^{*}t^{*y_{2}}s^{*x_{2}},~y=s^{x_{3}}t^{y_{3}}\upsilon_{c}\upsilon_{d}^{*}t^{*y_{4}}s^{*x_{4}}$ in $\mathcal{A}. $  Then we have
\begin{equation} \label{eq:KMS2}
a^{\beta}\phi(s^{x_{1}}t^{y_{1}}\upsilon_{a}\upsilon_{b}^{*}t^{*y_{2}}s^{*x_{2}}s^{x_{3}}t^{y_{3}}\upsilon_{c}\upsilon_{d}^{*}t^{*y_{4}}s^{*x_{4}})=b^{\beta}\phi(s^{x_{3}}t^{y_{3}}\upsilon_{c}\upsilon_{d}^{*}t^{*y_{4}}s^{*x_{4}}s^{x_{1}}t^{y_{1}}\upsilon_{a}\upsilon_{b}^{*}t^{*y_{2}}s^{*x_{2}})
\end{equation}
for $a,b,c,d \in \mathbb{N}^{\times},$  $m,n,q,r\in\mathrm{P},$ 
$m=2x_{1}+3y_{1},$
$n=2x_{2}+3y_{2},$
$q=2x_{3}+3y_{3},$ and
$r=2x_{4}+3y_{4}$
some $x_{i}, y_{i}$ in $\mathbb{N}$ for $i=1,2,3,4.$
We prove this equality by computing both sides.
To compute the left-hand side of (\ref{eq:KMS2}), we first reduce the formula by  using the covariance relation in Lemma \ref{le:0}  and  \ref{le:not 0} 
\begin{eqnarray*}
xy&=&(s^{x_{1}}t^{y_{1}}\upsilon_{a}\upsilon_{b}^{*}t^{*y_{2}}s^{*x_{2}})(s^{x_{3}}t^{y_{3}}\upsilon_{c}\upsilon_{d}^{*}t^{*y_{4}}s^{*x_{4}})\\
&=&s^{x_{1}}t^{y_{1}}\upsilon_{a}(\upsilon_{b}^{*}t^{*y_{2}}s^{*x_{2}}s^{x_{3}}t^{y_{3}}\upsilon_{c})\upsilon_{d}^{*}t^{*y_{4}}s^{*x_{4}}\\
&=& \begin{cases}
0 &if~ (n+b\mathrm{P})\cap(q+c\mathrm{P})=\emptyset,\\
s^{x_{1}}t^{y_{1}}\upsilon_{a}(s^{\alpha^{'}}t^{\alpha^{''}}\upsilon_{c^{'}}\upsilon_{b^{'}}^{*}t^{*\beta^{''}}s^{*\beta^{'}})\upsilon_{d}^{*}t^{*y_{4}}s^{*x_{4}}&if~(n+b\mathrm{P})\cap(q+c\mathrm{P})\neq\emptyset,
\end{cases} \end{eqnarray*}
where $k=(q-n)/gcd(b,c)=2k_{1}+3k_{2}$ some $k_{1},k_{2}$ in $\mathbb{Z}$,$~b^{'}=b/gcd(b,c),~c^{'}=c/gcd(b,c),$ and
$(\alpha,\beta)$ is the smallest non-negative solution of $k=b^{'}\alpha-c^{'}\beta,$
$\alpha\neq 1$, $ \beta \neq 1$, $\alpha= 2\alpha^{'}+3\alpha^{''}$, and  $ \beta=2 \beta^{'}+3 \beta^{''}$ some $\alpha^{'}, \alpha^{''}, \beta^{'}, \beta^{''}$ in $\mathbb{N}.$
By $\mathrm{(T1^{'})}$
\[xy =\begin{cases}
0 &if~ (n+b\mathrm{P})\cap(q+c\mathrm{P})=\emptyset,\\
s^{(x_{1}+a\alpha^{'})}t^{(y_{1}+a\alpha^{''})}\upsilon_{ac^{'}}\upsilon_{db^{'}}^{*}t^{*(y_{4}+d\beta^{''})}s^{*(x_{4}+d\beta^{'})}&if~(n+b\mathrm{P})\cap(q+c\mathrm{P})\neq\emptyset.\end{cases} \]
Now (\ref{eq:KMS}) implies that the left-hand side of (\ref{eq:KMS2}) is 
\begin{eqnarray} \label{(4.3)}
a^{\beta}\phi(xy) \!\!=\!\begin{cases}
0 &\!\!if (n+b\mathrm{P})\cap(q+c\mathrm{P})=\emptyset,\\
0 &\!\!if ac^{'}\neq db^{'}~or~m+a\alpha\not\equiv r+d\beta~\\
&(\!\!\!\!\!\mod~ac^{'}),\\
(c^{'})^{-\!\beta}\!\phi(t^{((\frac{y_{1}+a\alpha^{''}-y_{4}-d\beta^{''}+2u_{0}}{ac'}))}s^{((\frac{x_{1}+a\alpha^{'}-x_{4}-d\beta^{'}-3u_{0}}{ac'}))})&\!\!if~ac^{'}=db^{'},\\
&m+a\alpha\equiv r+d\beta~(\!\!\!\!\!\mod~ac^{'}), \\
&and~(n+b\mathrm{P})\cap(q+c\mathrm{P})\neq\emptyset\end{cases} 
\end{eqnarray}

where $x_{1}+a\alpha^{'}-x_{4}-d\beta^{'}\equiv3u_{0}~(\!\!\mod~ac^{'})$ and $y_{1}+a\alpha^{''}-y_{4}-d\beta^{''}\equiv-2u_{0}~(\!\!\mod~ac^{'}).$

The analogous computation shows that the right-hand side of (\ref{eq:KMS2}) is 
\begin{eqnarray} \label{(4.4)}
b^{\beta}\phi(yx) =\begin{cases}
0 &\!\!if (r+d\mathrm{P})\cap(m+a\mathrm{P})=\emptyset,\\
0 &\!\!if~ ca^{'}\neq bd^{'}~or~q+c\gamma\not\equiv n+b\delta~\\
&(\!\!\!\!\!\mod~bd^{'}),\\
(d^{'})^{-\beta}\phi(t^{((\frac{y_{3}+c\gamma^{''}-y_{2}-b\delta^{''}+2v_{0}}{bd'}))}s^{((\frac{x_{3}+c\gamma^{'}-x_{2}-b\delta^{'}-3v_{0}}{bd'}))})&\!\!if~ca^{'}=bd^{'},\\
&q+c\gamma\equiv n+b\delta~(\!\!\!\!\!\mod~bd^{'}), \\
&and~(r+d\mathrm{P})\cap(m+a\mathrm{P})\neq\emptyset\end{cases} 
\end{eqnarray}
where $x_{3}+c\gamma^{'}-x_{2}-b\delta^{'}\equiv3v_{0}~(\!\!\mod~bd^{'}),$  $y_{3}+c\gamma^{''}-y_{2}-b\delta^{''}\equiv-2v_{0}~(\!\!\mod~bd^{'}),$ $k^{'}=(m-r)/gcd(a,d)=2k_{1}^{'}+3k_{2}^{'}$ some $k_{1}^{'},k_{2}^{'}$ in $\mathbb{Z},$
$~d^{'}=d/gcd(a,d),~a^{'}=a/gcd(a,d),$~and~$(\gamma,\delta)$ is the smallest non-negative solution of $k^{'}=d^{'}\gamma-a^{'}\delta,~\gamma\neq 1,~\delta \neq 1,~\gamma=2\gamma^{'}+3\gamma^{''},$~and~$\delta=2\delta^{'}+3\delta^{''}$ some $\gamma^{'},\gamma^{''},\delta^{'},\delta^{''}$ in $\mathbb{N}.$

We need to verify that the conditions  of   (\ref{(4.3)})  for  the  nonvanishing  case  match   those  of (\ref{(4.4)}). Since   the  situation is symmetric,  we  can   suppose that $ac^{'}=db^{'},$   $m+a\alpha \equiv r+d\beta~(mod~ac^{'}),$ and  $(n+b\mathrm{P})\cap(q+c\mathrm{P})\neq\emptyset$   where $(\alpha,\beta)$ is   defined  as  above.  Suppose first that $ac^{'}=db^{'}.$     We  see  that
\begin{eqnarray*}
ac^{'}=db^{'}&\Leftrightarrow a/d=b^{'}/c^{'} \Leftrightarrow& a^{'}/d^{'}=b^{'}/c^{'}\\
&\Leftrightarrow a^{'}/d^{'}=b/c \Leftrightarrow& ca^{'}=bd^{'};\end{eqnarray*}
Since $gcd( a^{'},d^{'})=1$    and  $gcd(b^{'},c^{'})=1$,  these are all equivalent to $ac=bd.$  And   we deduce that $a^{'}=b^{'}$ and $c^{'}=d^{'}   $  from the reduced form in the middle.
This implies that $(c^{'})^{-\beta}=(d^{'})^{-\beta}.$

Next, notice that  
$m-r \equiv d\beta-a\alpha~(\!\!\!\! \mod~ac^{'})$ implies that there is  some $u$ in $\mathbb{Z}$  such  that 
  $m-r=d \beta-a \alpha+ac^{'}u.$
 Let $\mathrm{G}=gcd(d,a),$ then $a=a^{'}\mathrm{G}$ and  $d=d^{'}\mathrm{G}.$  Since  
 $m-r=d^{'}\mathrm{G}\beta-a^{'}\mathrm{G}\alpha+a^{'}\mathrm{G}c^{'}u,$
   we  have    $m\equiv r~(\!\!\!\! \mod gcd(d,a)).$  Therefore $(r+d\mathrm{P})\cap(m+a\mathrm{P})\neq\emptyset .$
The definition of $(\gamma,\delta)$ implies that $m-r=d\gamma-a\delta. $  
 By using $a^{'}=b^{'}$ and  $c^{'}=d^{'},$ we have 
\begin{eqnarray} \label{=}
\frac{m+a\alpha-r-d\beta}{ac^{'}}&
=&\frac{d\gamma-a\delta+a\alpha-d\beta}{ac^{'}}\nonumber \\
&=&\frac{(\gamma-\beta)d+(\alpha-\delta)a}{gcd(d,a)a^{'}c^{'}}\nonumber\\
&=&\frac{(\gamma-\beta)d^{'}+(\alpha-\delta)a^{'}}{a^{'}c^{'}}\nonumber\\
&=&\frac{(\gamma-\beta)c^{'}+(\alpha-\delta)b^{'}}{b^{'}d^{'}}\nonumber\\
&=&\frac{(q-n)+c\gamma-\delta b}{bd^{'}}.\end{eqnarray}
Therefore   $ac^{'}|(m+a\alpha-r-d\beta)$ shows that 
$bd^{'}|(q-n+c\gamma-\delta b),$ or equivalently that $q+c\gamma\equiv n+b\delta~(mod~bd^{'}).$ Since the conditions for a vanishing  (\ref{(4.3)}) and (\ref{(4.4)}) are  the contraposition of the conditions for a nonvanishing, all conditions of  (\ref{(4.3)}) and (\ref{(4.4)}) were shown to be in agreement.

Consider now the exponent of $t$ and $s$ of (\ref{(4.3)}) and (\ref{(4.4)}). Substituting $m=2x_{1}+3y_{1},$ $n=2x_{2}+3y_{2},$ $q=2x_{3}+3y_{3},$ $r=2x_{4}+3y_{4},$ $\alpha=2\alpha'+3\alpha'',$ $\beta=2\beta'+3\beta'',$ $\gamma=2\gamma'+3\gamma'',$ and $\delta=2\delta'+3\delta''$ into (\ref{=}),  we have 
$$2(x_{1}+a\alpha^{'}-x_{4}-d\beta^{'})+3(y_{1}+a\alpha^{''}-y_{4}-d\beta^{''})=2(x_{3}+c\gamma^{'}-x_{2}-b\delta^{'})+3(y_{3}+c\gamma^{''}-y_{2}-b\delta^{''}).$$
Since $x_{1}+a\alpha^{'}-x_{4}-d\beta^{'}=3u_{0}+Xac'$ some $X \in \mathbb{Z},$ $y_{1}+a\alpha^{''}-y_{4}-d\beta^{''}=-2u_{0}+Yac'$ some $Y \in \mathbb{Z},
$ $x_{3}+c\gamma^{'}-x_{2}-b\delta^{'}=3v_{0}+X'bd'$ some $X' \in \mathbb{Z},$ and $y_{3}+c\gamma^{''}-y_{2}-b\delta^{''}=-2v_{0}+Y'bd'$ some $Y' \in \mathbb{Z},$ we have $2(X-X')=3(Y'-Y).$ We can take $X'=X-3z$ and $Y'=Y+2z$ some $z \in \mathbb{Z}.$ Therefore
\begin{eqnarray*}
t^{((\frac{y_{3}+c\gamma^{''}-y_{2}-b\delta^{''}+2v_{0}}{bd'}))}s^{((\frac{x_{3}+c\gamma^{'}-x_{2}-b\delta^{'}-3v_{0}}{bd'}))}&=&t^{((Y'))}s^{((X'))}\\
&=&t^{((Y))}t^{((2z))}s^{*((3z))}s^{((X))}\\
&=&t^{((Y))}s^{((X))}\\
&=&t^{((\frac{y_{1}+a\alpha^{''}-y_{4}-d\beta^{''}+2u_{0}}{ac'}))}s^{((\frac{x_{1}+a\alpha^{'}-x_{4}-d\beta^{'}-3u_{0}}{ac'}))}.
\end{eqnarray*}
It shows that a nonvanishing value of $a^{\beta}\phi(xy)$ equals a nonvanishing value of $b^{\beta}\phi(yx)$ where $x=s^{x_{1}}t^{y_{1}}\upsilon_{a}\upsilon_{b}^{*}t^{*y_{2}}s^{*x_{2}},~y=s^{x_{3}}t^{y_{3}}\upsilon_{c}\upsilon_{d}^{*}t^{*y_{4}}s^{*x_{4}}$ in $\mathcal{A}$ for (\ref{(4.3)}) and (\ref{(4.4)}).
This completes the proof of (\ref{eq:KMS2}), and we show that $\phi$ is a $\mathrm{KMS}_{\beta}$ state.\end{proof}

\section*{Acknowledgements}
  { S. Y. Jang  was supported by Basic Science Research Program through the
National Research Foundation of Korea funded by the Ministry of
Education, Science and Technology (NRF-201807042748). }


\end{document}